\documentclass[a4paper,11pt]{article}
\usepackage{hyperref}
\usepackage{geometry}
\usepackage[british]{babel}
\usepackage{wrapfig}
\usepackage[comma, square, numbers]{natbib}
\usepackage{yfonts}
\usepackage[utf8]{inputenc}
\usepackage{amssymb}
\usepackage{amsthm}
\usepackage{graphics}
\usepackage{amsmath}
\usepackage{amstext}
\usepackage{engrec}
\usepackage{floatflt}
\usepackage{rotating}
\usepackage[safe,extra]{tipa}
\usepackage[font={footnotesize}]{caption}
\usepackage{framed}
\usepackage{listings}
\usepackage{subcaption}
\usepackage[dvipsnames]{xcolor}

\usepackage{mathtools}

\usepackage[titletoc,title]{appendix}

\newcommand{\mc}{\mathcal}
\newcommand{\mb}{\mathbb}

\newcommand{\R}{\mb R}
\newcommand{\C}{\mb C}
\newcommand{\N}{\mb N}

\newcommand{\T}{\mb T}

\newcommand{\eea}{\end{align}}

\renewcommand{\epsilon}{\varepsilon}
\renewcommand{\bar}{\overline}
\renewcommand{\tilde}{\widetilde}

\newcommand{\bo}{\boldsymbol}
\renewcommand{\phi}{\varphi}

\renewcommand\upsilon{\theta}

\usepackage[normalem]{ulem}

\newtheorem{theorem}{Theorem}[section]

\newtheorem{corollary}{Corollary}[section]
\newtheorem{lemma}{Lemma}[section]

\newtheorem{proposition}{Proposition}[section]
\theoremstyle{definition}
\newtheorem{definition}{Definition}[section]

\theoremstyle{remark}
\newtheorem{remark}{Remark}[section]

\newtheoremstyle{algorithm}
{4pt}
{4pt}
{}
{}
{}
{:}
{\newline}
{}

\usepackage{xcolor}

\newtheorem{algorithm}{Algorithm}

\newcommand{\balgorithm}{\begin{algorithm}\begin{framed}\ }
\newcommand{\ealgorithm}{\end{framed}\end{algorithm}}

\newcommand{\pippo}{\begin{code} \  \begin{lstlisting}[frame=single]}
\newcommand{\pappo}{\end{lstlisting} \end{code}}

\newcommand{\bd}{\begin{definition}}
\newcommand{\ed}{\end{definition}}

\newcommand{\bt}{\begin{theorem}}
\newcommand{\et}{\end{theorem}}
\newcommand{\bp}{\begin{proposition}}
\newcommand{\ep}{\end{proposition}}
\newcommand{\bc}{\begin{corollary}}
\newcommand{\ec}{\end{corollary}} 
\newcommand{\bl}{\begin{lemma}}
\newcommand{\el}{\end{lemma}}
\newcommand{\br}{\begin{remark}}
\newcommand{\er}{\end{remark}}

\DeclareMathOperator{\Lip}{Lip}

\theoremstyle{definition}

\theoremstyle{remark}

\title{Linear Response for a Family of Self-Consistent Transfer Operators}
\author{Fanni M. S\'elley\footnote{Laboratoire de Probabilit\'es, Statistique et Mod\'elisation (LPSM), CNRS,
		Sorbonne Universit\'e, Universit\'e de Paris, 4 Place Jussieu, 75005 Paris, France, \textit{email:} selley@lpsm.paris (corresponding author). The research of F. M. Sélley was supported
by the E\-u\-ro\-pe\-an Research Council (ERC) under the European Union's Horizon 2020 research and innovation programme (grant agreement No 787304). }, Matteo Tanzi\footnote{Courant Institute of Mathematical Sciences, NYU, 251 Mercer St., New York, NY, USA, \textit{email:} matteo.tanzi@sns.it }}
\date{}

\begin{document}
\maketitle

\begin{abstract}
We study a system of globally coupled uniformly expanding circle maps in the thermodynamic limit. The state of the system is described by a probability density and its evolution is given by the action of a nonlinear operator called the self-consistent transfer operator. Self-consistency is understood in the following sense: if $\varphi$ corresponds to the system's state, then the evolution of $\varphi$ is given by the application of the transfer operator of the circle map $F_{\varphi}$. This is a $\phi-$dependent map describing the dynamics of a single unit in the finite system where the interaction coming from all the nodes is replaced with the average of the interaction term with respect to the density $\phi$. Assuming some level of smoothness of the coupled circle maps and of the coupling, we prove that when the coupling strength is sufficiently small, the system has a unique smooth stable state. We then show that this stable state is continuously differentiable as a function of the coupling strength. Finally, we prove that the derivative satisfies a linear response formula which at zero simplifies to a formula reminiscent to the one obtained for the linear transfer operators of perturbed expanding circle maps. 
\end{abstract}

\section{Introduction}

A fundamental question in the theory of dynamical systems is how the statistical properties of a system change when it is subjected to perturbation. A system exhibits linear response if the invariant measure depends smoothly on the perturbation, and an expression for the derivative of the invariant measure (in the strong  or in the dual sense) is called  a linear response formula.

The rigorous study of linear response dates back to Ruelle, who proved linear response of uniformly hyperbolic Axiom A systems \cite{ruelle1997differentiation,ruelle1998general,ruelle2009review}. Similar results can be proved in some non-uniformly hyperbolic \cite{dolgopyat2004differentiability,zhang2018smooth} or non-uniformly expanding cases such as intermittent maps \cite{bahsoun2015linear,baladi2016linear,korepanov2016linear} and piecewise expanding unimodal maps \cite{baladi2008linear,baladi2010alternative}. But caution is required, as examples for the lack of linear response are also well known \cite{aspenberg2019fractional,baladi2007susceptibility,baladi2015whitney,baladi2017linear} (see \cite{baladi2014linear} for a more comprehensive collection of references). Deterministic systems subjected to random perturbations are also known to exhibit linear response in some cases \cite{galatolo2019linear,galatolo2019quadratic}. In addition to theoretical works, linear response theory is successfully applied in geophysics, in particular climate science \cite{lucarini2014mathematical,majda2012challenges}.

In this paper we rigorously prove a linear response formula for  all-to-all weakly coupled uniformly expanding maps in the thermodynamic limit. To the best of our knowledge, there are few results available on linear response of complex systems with multiple interacting units. A program focusing on the control of statistical properties of extended systems by exploiting linear response was outlined by Mackay \cite{mackay2018management}. This question was further pursued in the case of a single dynamical unit defined by a uniformly expanding circle map \cite{galatolo2017controlling} and by a smooth map of a compact smooth Riemannian manifold \cite{kloeckner2018linear}. In case of coupled map lattices, the smooth dependence of the SRB measure and the linear response formula for  was studied by Jiang and de la Llave \cite{jiang2006linear,jiang1999smooth}. Wormell and Gottwald \cite{wormell2018validity,wormell2019linear} recently showcased numerical examples  and some rigorous arguments for all-to-all coupled maps with mean-field interaction showing that it is possible for a high dimensional compound system to exhibit linear response, even if its  units do not.        

 Our setup is based on the model studied by Fernandez \cite{fernandez2014breaking}. The finite size model is described by $M$ coordinates $(x_1,...,x_M)$ with $x_i\in\T$, whose evolution is given by 
\begin{equation}\label{Eq:FiniteDyna}
x_i(t+1)=f\circ \Phi_{\epsilon}\left(x_i(t);x_1(t),...,x_M(t)\right)
\end{equation}
where $\Phi_\epsilon:\T\times \T^M\rightarrow \T$,
\begin{equation}\label{Eq:CoupExp}
\Phi_{\epsilon}\left(x_i;x_1,...,x_M\right)=x_i+ \frac{\epsilon}{M}\sum_{j=1}^Mh\left(x_i,x_j\right),
\end{equation}
is a smooth mean-field all-to-all coupling between the coordinates ($\epsilon$ is the coupling strength\footnote{For example, when $\epsilon=0$, $\Phi_0$ equals the identity and the evolution of any coordinate $i$ equals the uncoupled dynamics $f$.}), and $f:\T\rightarrow \T$ is a sufficiently smooth map describing the uncoupled dynamics.

When taking the thermodynamic limit, $M\rightarrow \infty$,  one can replace the finite dimensional state space $\T^M$ with the infinite-dimensional state space $\T^\N$, and the coupling term in \eqref{Eq:CoupExp} with
\[
\lim_{M\rightarrow \infty}\frac{\epsilon}{M}\sum_{j=1}^Mh\left(x_i,x_j\right).
\]
The limit above might not  exist. Let's assume that for the initial condition $(x_i(0))_{i=1}^\infty\in \T^\N$ there is a probability measure $\mu$  such that    
 \begin{equation}\label{Eq:AssumConvWeak}
\lim_{M\rightarrow\infty} \frac{1}{M}\sum_{i=1}^M\delta_{x_i(0)}= \mu
 \end{equation}
  where the limit is with respect to the  weak convergence of measures \footnote{Notice that the existence of the limit and value of the limit in  \eqref{Eq:AssumConvWeak} depends, in general, on the ordering of the coordinates.
  
  Given a sequence of i.i.d. random variables $\{X_i\}_{i\in\N}$ such that the law of $X_i$ is given by $\mu$, then almost every realization of the sequence $\{X_i\}_{i\in\N}$ satisfies \eqref{Eq:AssumConvWeak}. This is a consequence of the strong Law of Large Numbers and the separability of continuous functions from $S^1$ to $\R$ with respect to the topology induced by the sup norm.}.  This assumption ensures that  for any continuous $h$, 
 \[
 \lim_{M\rightarrow\infty}\frac{\epsilon}{M}\sum_{j=1}^Mh(x_i(0),x_j(0))= \epsilon\int h(x_i(0),y)d\mu(y),
 \]
 and the time evolution of each coordinate can be written as:
 \[
 x_i(1)=f\circ \Phi_{\epsilon,\mu}(x_i(0))\quad\mbox{where}\quad \Phi_{\epsilon,\mu}(x)=x+\epsilon \int h(x,y)d\mu(y).
 \]  
 Calling $f_*$ and ${(\Phi_{\epsilon,\mu}})_*$ the push-forwards of $f$ and ${\Phi_{\epsilon,\mu}}$,
 \[
 \lim_{M\rightarrow \infty}\frac{1}{M}\sum_{i=1}^M\delta_{x_i(1)}=f_*{(\Phi_{\epsilon,\mu}})_*(\mu)
 \]
 where the convergence is weak and follows from the continuity of the push forwards (which only requires the continuity of $f$ and $\Phi_{\epsilon,\mu}$). In case $\mu$ has a density $\phi$ with respect to a suitable reference measure (Lebesgue, for instance) it is sufficient to study the evolution of the density. By the above argument, this is given by a transfer operator that depends on the probability measure, thus $\phi$ itself (for a precise definition, see \eqref{Eq:TransfOp}). This nonlinear application is called a \emph{self-consistent transfer operator. This will be the object under study in this paper.} 

Fernandez \cite{fernandez2014breaking} studied the finite setting in the case of $h(x,y)=g(x-y)$ with $g: \mathbb{R} \to \mathbb{R}$ defined to be a specific discontinuous, but piecewise affine function. Studying the thermodynamic limit of his setup in case of $\varepsilon$ sufficiently small, Bálint and Sélley \cite{selley2016mean} first showed that when $f(x)=2x \mod 1$, the constant function is the only smooth fixed density. Expanding these results, in \cite{balint2018synchronization}, they considered a general uniformly expanding smooth circle map $f$, and showed that the invariant density is unique and  a Lipschitz continuous function of the coupling strength, with higher regularity being obstructed by the discontinuity of the function $h$.

In this paper we present an approach based on convex cones of functions to study self-consistent operators for general uniformly expanding $f$ and general $h$ satisfying certain smoothness conditions. We show that in the weak coupling regime ($\epsilon$ small), the self-consistent operator has a unique fixed smooth density that depends in a differentiable way on the coupling strength. Furthermore, we compute a linear response formula for the density. 

We believe that assuming stronger smoothness conditions on $h$ and $f$, higher smoothness of the invariant density in terms of the coupling could be also achieved and a quadratic response formula could be deduced, reminiscent of the formula obtained in \cite{galatolo2019quadratic}. Another, more difficult, open question is whether one could apply ideas similar to the one in this paper to treat the case where $h$ is not smooth. This would require the definition of some suitable Banach spaces and cones playing analogous roles to the spaces and cones of smooth functions that we consider here.
 
 \subsection{Setup}\label{Sec:Setup}
 \paragraph{Function spaces and norms} Let $S^1$ denote the unit circle and denote the Lebesgue (Haar) measure on $S^1$ by $\lambda$. Throughout this paper, we consider the representation of $S^1$ as the unit interval $[0,1]$ with 0 and 1 identified. We denote by $L^1(S^1,\mathbb{R})$ the real valued functions on $S^1$ for which
 \[
 \int_{S^1}|f|\:d\lambda < \infty
 \]
 holds. Let $C^0(S^1,S^1)$ denote the space of continuous mappings of $S^1$, endowed with the norm
 \[
 \|f\|_{C^0}=\max_{x\in S^1}|f(x)|.
 \]
 For integers $k \geq 1$, let $C^k(S^1,S^1)$ denote the space of $k$-times continuously differentiable mappings of $S^1$, endowed with the norm
 \[
 \|f\|_{C^k}=\max_{x\in S^1}|f(x)|+\sum_{\ell=1}^k\max_{x\in S^1}|f^{(k)}(x)|.
 \]
\paragraph{Uncoupled Dynamics} Let $f \in C^k(S^1,S^1)$ (with $k>1$ to be specified later) be a uniformly expanding circle map, i.e. there exists an $\omega > 1$ such that $|f'(x)| > \omega$ for all $x \in S^1$, and denote by $P$ its transfer operator with respect to the Lebesgue measure. That is, $P: L^1(S^1,\mathbb{R}) \to L^1(S^1,\mathbb{R})$ is defined as
\begin{equation} \label{Eq:TransferOpP}
P\varphi(x)=\sum_{y \in f^{-1}(x)}\frac{\varphi(y)}{|f'(y)|}, \qquad \varphi \in L^1(S^1,\mathbb{R}).
\end{equation}
\paragraph{Mean-Field Coupling} Given $\epsilon_0>0$, $t\in[-\epsilon_0,\epsilon_0]$, $h \in C^k(S^1 \times S^1, \mathbb{R})$ (with $k>1$ to be specified later), and $\psi \in L^1(S^1,\mathbb{R}_0^+)$\footnote{$\R^+_0$ is the set containing the positive real numbers together with 0.}  with $\int_{S^1} \psi =1$,  define $\Phi_{t,\psi}:S^1\rightarrow S^1$
\[
\Phi_{t,\psi}(x)=x+t\int_{S^1} h(x,y)\psi(y)dy.
\]
We assume that $\epsilon_0>0$ is sufficiently small so that $\Phi_{t,\psi}$ is a diffeomorphism  such that $\Phi_{t,\psi}' > 0$, and we denote by $Q_{t,\psi}$ its transfer operator. More precisely, similarly to \eqref{Eq:TransferOpP}, the operator $Q_{t,\psi}: L^1(S^1,\mathbb{R}) \to L^1(S^1,\mathbb{R})$ is defined as
\begin{equation} \label{Eq:TransferOpQ}
Q_{t,\psi}\varphi(x)=\frac{\varphi}{\Phi_{t,\psi}'}\circ \Phi_{t,\psi}^{-1}(x), \qquad \varphi \in L^1(S^1,\mathbb{R}).
\end{equation}
The evolution of the density of states $\psi$  under the mean-field coupling is then given by the nonlinear application of $Q_t$ defined as $Q_t(\psi):=Q_{t,\psi}\psi$.
In the following we use the short-hand notation $A_{\psi}(x)=\int_{S^1} h(x,y)\psi(y)dy$, so we write
$
\Phi_{t,\psi}(x)=x+tA_{\psi}(x).
$
Thus substituting $\varphi=\psi$ in \eqref{Eq:TransferOpQ} gives the following expression for $Q_t$:
\begin{equation}\label{Eq:QepsDef}
Q_{t}\psi(x)=\frac{\psi}{1+tA'_{\psi}}\circ \Phi^{-1}_{t,\psi}(x).
\end{equation}

\paragraph{\textbf{Coupled Dynamics}} Define $F_{t,\psi}:=f\circ \Phi_{t,\psi}$, and let $\mc L_{t,\psi}:=PQ_{t,\psi}$ denote its transfer operator. Given $t\in[-\epsilon_0,\epsilon_0]$, define the self-consistent transfer operator  $\mc L_t:L^1(S^1,\mathbb{R})\rightarrow L^1(S^1,\mathbb{R})$ as
\begin{equation} \label{Eq:TransfOp}
\mc L_t:=PQ_t
\end{equation}
which describes the evolution of the all-to-all coupled system in the thermodynamic limit.

\paragraph{Smooth densities} Throughout the paper we are going to work with the restrictions of $P,Q_t$ and $\mc L_t$ to some appropriate spaces of smoothly differentiable functions. These restrictions are well-defined: on the one hand, $\varphi \in C^{\ell}$ implies $P\varphi \in C^{\ell}$ for $\ell < k$ supposing $f \in C^k(S^1,S^1)$. On the other, if $h \in C^{k'}(S^1 \times S^1,\mathbb{R})$ then $\Phi_{t,\psi} \in C^{k'}(S^1,S^1)$ for sufficiently small values of $t$, thus $Q_t\varphi \in C^{\ell}$ for $\ell < k'$.

By continuity, $f$ is an $N$-fold covering map of $S^1$ for some $N \in \mathbb{N}$. We  make a technical assumption on $f$ ensuring that the expansion is large enough and the distortion is small enough, more precisely
\begin{equation} \label{Eq:Assumf}
N\left(\frac{\max|f''|}{\omega^3}+\frac{1}{\omega^2}\right) < 1.
\end{equation}
Note that for example the map $f(x)=kx \mod 1$ (where $k$ is an integer) and moderate perturbations satisfy this condition. We remark that we could exclude this condition by working with higher iterates of the self-consistent transfer operator. However, this would complicate the calculations a great deal and we will assume \eqref{Eq:Assumf} for the sake of clarity of the presentation.
 From now on, $K>0$ will denote an unspecified positive number that depends on $h$ and bounds of its derivatives only, while $\tilde K>0$ is another unspecified constant depending also on $f$, unless otherwise specified. 
\subsection{Summary of the results and outline of the paper} 
The first theorem of this paper claims that for $t$ sufficiently small, the system has a unique equilibrium density $\rho(t)$, and any sufficiently smooth initial state will evolve towards $\rho(t)$.
\begin{theorem}\label{Thm:ExistLowReg}
Assume $f\in C^2(S^1,S^1)$ and $h\in C^2(S^1\times S^1,\mathbb{R})$. Then, there is $\epsilon_0>0$ sufficiently small such that for any $t\in[-\epsilon_0,\epsilon_0]$, $\mc L_t$ has a unique fixed non-negative probability density $\rho(t)\in \Lip(S^1,\R^+)$. Furthermore, there exist constants $\tilde{K}>0$ and $\Lambda<1$ such that for any non-negative $\phi\in\Lip (S^1,\R)$with $\int \varphi =1$ it holds that \[\|\mc L_t^n\phi-\rho(t)\|_{C^0}\leq \tilde{K}\Lambda^n \qquad \text{for all $n \in \mathbb{N}$.}\] 
\end{theorem}

\begin{remark}
A very similar result was stated by \cite[Theorem 4]{keller2000ergodic}. However the coupling mechanism there is different in the sense that an integral average is computed from the current system statistic which feeds into the dynamics as a constant. Our coupling is more complicated: it varies from point to point due to the $x$-dependence of the function $h$. However, the results and the methods that can be used to prove them are essentially the same. We view this statement as a preparatory result for the next theorem, our main result.     
\end{remark}
The second theorem claims that taking $f$ and $h$ sufficiently smooth, the dependence of $\rho (t)$ on $t$ is $C^1$ in a strong sense. 
\begin{theorem}\label{Thm:LinResp}
Assume $f\in C^5(S^1,S^1)$ and $h\in C^5(S^1\times S^1,\mathbb{R})$. Then\begin{itemize}
\item[(i)] there is an $\epsilon_1 > 0$, $\varepsilon_1 \leq \varepsilon_0$ such that $\rho(t)\in C^3(S^1,\R)$ for all $t \in [-\varepsilon_1,\varepsilon_1]$;
\item[(ii)] there is an $\epsilon_2 > 0$, $\varepsilon_2 \leq \varepsilon_1$ such that $t\mapsto \rho(t)$ belongs to $C^1([-\epsilon_2,\epsilon_2],C^1(S^1,\R))$;
\item[(iii)] there is an $\epsilon_3 > 0$, $\varepsilon_3 \leq \varepsilon_2$ such that the following linear response formula holds for $|\hat{t}| < \varepsilon_3$:
\[
\partial_{t}\rho(t)|_{t=\hat{t}}=-(1-P_{\hat{t}}+\hat{t}P\Kappa_{\hat{t}})^{-1}P\Kappa_{\hat{t}}(\rho(\hat{t})),
\]
where $P_{\hat{t}}=PQ_{\hat{t},\rho(\hat{t})}$ and the operator $\Kappa_{\hat{t}}$ is given by \eqref{Eq:ThetaOp}.
\end{itemize}  

\end{theorem}

\begin{remark}
	Both theorems depend intimately on the fact that when the coupling strength is sufficiently weak, the operators satisfy uniform Lasota-Yorke type inequalities implying uniform spectral properties. However, standard perturbation theory arguments are insufficient to obtain the results above due to the nonlinear character of the operators $\mathcal L_t$. Getting around the nonlinearity of the operators is where the main novelty of our approach lies.    
\end{remark}
\noindent We break down the proof of the above theorems into several propositions listed below and proved in the following sections.

In Section \ref{Sec:InvDens} we study existence, uniqueness and regularity of the fixed point for the self-consistent transfer operator $\mc L_t$ \eqref{Eq:TransfOp}. To this end, we consider the cone of log-Lipschitz functions
\[
\mc V_a:=\left\{\phi:S^1\rightarrow \R^+:\;\frac{\phi(x)}{\phi(y)}\le \exp(a|x-y|)\right\},
\]
where $|x-y|$ denotes the Euclidean distance between $x$ and $y$ on the unit circle, and show that for large enough $a$, $\mc L_t$ keeps $\mc V_a$ invariant. Most importantly, we can also show that $\mc L_t$  is a contraction with respect to the Hilbert metric $\theta_a$ (see \eqref{Eq:HilbertMetric} below) on $\mc V_a$. 

  \begin{remark}
Cones of functions and the Hilbert metric are common tools to prove uniqueness of the invariant density for some linear transfer operators \cite{liverani1995decay}. The standard approach uses the fact that a linear operator mapping a cone with finite diameter strictly inside itself is a contraction in the Hilbert metric. Unlike the standard case,  $\mc L_t$ is a nonlinear operator, and one has to prove explicitly that $\mc L_t$ is a contraction. 
\end{remark}
\begin{proposition}\label{Lem:ConeCont}
There is $a_0>0$ sufficiently large such that for $a>a_0$, there is  $\epsilon_0>0$ sufficiently small such that for $t\in[-\epsilon_0,\epsilon_0]$, $\mc L_t$ keeps $\mc V_a$ invariant. Furthermore, there are $a>0$, and $\epsilon_0>0$ small such that $\mc L_t$ is a contraction on $(\mc V_a,\theta_{a})$ for any $t\in[-\epsilon_0,\epsilon_0]$.
 \end{proposition}
 \noindent 
 The proof of Theorem \ref{Thm:ExistLowReg} then follows by standard arguments showing that under repeated application of $\mc L_t$, all the elements in the cone $\mc V_a$ converge exponentially fast to the fixed point $\rho(t)$ (in the $C^0$ norm). 

To show that the $C^5$-smoothness of $f$ and $h$ implies that $\rho(t) \in C^3(S^1,\mathbb{R})$, define the following set for the quadruple $\bo C=(C_1,C_2,C_3,C_4)\in (\R^+)^4$:
\[
\mc C_{\bo C}:=\left\{\phi\in C^4(S^1,\R^+):\;\int \phi=1,\; \left\|\frac{d^i\phi}{dx^i}\right\|_{\C^0}<C_i\,\;i=1,...,4\right\}.
\]
We show that
\begin{proposition}\label{Prop:InvSetDens}
There is $\epsilon_1>0$ and  $C_1,C_2,C_3,C_4>0$ such that for $t\in[-\epsilon_1,\epsilon_1]$
\[
\mc L_t \mc C_{\bo C}\subset \mc C_{\bo C}.
\]
\end{proposition}
\noindent  The existence of a fixed point in $C^3$ is then implied by the fact that $\mc L_t$ is continuous (Lemma \ref{Prop:ContSelfconOperator} below),  $\mc C_{\bo C}$ is relatively compact in $C^3$, and the Schauder fixed-point theorem.  Since one can find $\phi\in\mc C_{\bo C} \cap \mc V_a$ (for example take $\phi=1$), the fixed point $\rho \in\mc V_a$, whose existence was argued above, must belong to the closure of $\mc C_{\bo C}$ and therefore is in $C^3$, that will prove point (i) of Theorem \ref{Thm:LinResp}.

In Section \ref{Sec:SmoothDep} we show that, when $f$ and $h$ are $C^5$, the curve of fixed points $t\mapsto \rho(t)$ found above is $C^1$  from $[-\epsilon_2,\epsilon_2]$ to $C^1(S^1,\R)$ for some $\epsilon_2>0$. 
\begin{remark} In the case of a one-parameter family $ \{f_t\}_t$ of sufficiently smooth uniformly expanding maps, the strong differentiability in $t$ follows from the spectral properties of the associated transfer operators and the smoothness of $t\mapsto f_t$ \cite{baladi2014linear}. Here, since the one-parameter family depends on the fixed density itself, i.e. $\rho(t)$ is the fixed density of $f_t=F_{t,\rho(t)}$ (self-consistent relation), the usual arguments do not apply since there is no a priori knowledge of the regularity of $t\mapsto f_t$ that has to be established first. 
\end{remark}
This is achieved by showing that the set of curves 
\[
\hat {\mc C}_{\epsilon_2,K_1,K_2}:=\left\{ \gamma:[-\epsilon_2, \epsilon_2]\rightarrow \mc C_{\bo C}\mbox{ s.t. }\;\substack{(t,x)\mapsto \gamma(t)(x) \mbox{ is } C^2\\ \sup_t\left\|\frac{d}{dt}\gamma(t)\right\|_{C^3}<K_1, \;\sup_t\left\|\frac{d^2}{dt^2}\gamma(t)\right\|_{C^2}<K_2}\right\}
\]
is invariant for suitably chosen $K_1$ and $K_2$ greater than zero under the mapping  \[(\mc L\gamma)(t)=\mc L_t(\gamma (t))\mbox{ for all }t\in [-\epsilon_2,\epsilon_2]\] 
naturally defined by the family of operators on a curve.
\begin{proposition}\label{Prop:ChatInv}
Choose $C_1,C_2,C_3,C_4>0$ so that Proposition \ref{Prop:InvSetDens} holds. Then there is $\epsilon_2>0$ sufficiently small and $K_1,K_2>0$ such that \[\mc L \hat {\mc C}_{\epsilon_2,K_1,K_2}\subset  \hat {\mc C}_{\epsilon_2,K_1,K_2}.\]
\end{proposition}
\noindent   Since $\hat {\mc C}_{\epsilon_2,K_1,K_2}$ is relatively compact in $C^1([-\epsilon_2,\epsilon_2],C^1(S^1,\R))$, for any $\gamma\in\hat{\mc C}$, there is an increasing sequence $\{n_k\}_{k\in\N}$ and $\bar \gamma\in C^1([-\epsilon_2,\epsilon_2],C^2(S^1,\R))$ such that  $\mc L^{n_k}\gamma\rightarrow \bar \gamma$, and since by  Proposition \eqref{Prop:ChatInv}  $\mc L^{n_k}_t(\gamma(t))\rightarrow \rho(t)$ for every $t$, $\bar\gamma(t)= \rho(t)$. This proves point (ii) of Theorem \ref{Thm:LinResp}.

In Section \ref{Sec:LinRespFormula}, we prove point (iii) of Theorem \ref{Thm:LinResp} by calculating a linear response formula. To ease the notation, we denote $P_t:=\mc L_{t,\rho(t)}$.
\begin{proposition}\label{Prop:LinearResp}
	We have the formula
	\begin{equation}\label{Eq:Lineart0}
\partial_t\rho|_{t=0}=-(1-P)^{-1}P(\rho(0) A_{\rho(0)})',
\end{equation}
where $\partial_t$ is understood as $\partial_t: C^1([-\varepsilon_3,\varepsilon_3],C^1(S^1,\mathbb{R})) \mapsto C^0([-\varepsilon_3,\varepsilon_3],C^1(S^1,\mathbb{R}))$.
More generally, for any $\hat t\in (-\epsilon_3,\epsilon_3)$ where $\varepsilon_3 \leq \varepsilon_2$ as in Proposition \ref{Prop:ChatInv},
	\[
	\partial_{t}\rho|_{t=\hat{t}}=-(1-P_{\hat{t}}+\hat{t}P\Kappa_{\hat{t}})^{-1}P(\Kappa_{\hat{t}}(\rho(\hat{t})))
	\]
	where $P_{\hat{t}}=PQ_{\hat{t}\rho(\hat{t})}$ and
	\begin{equation} \label{Eq:ThetaOp}
	\Kappa_{t}(g)=\left(\left(\frac{\rho(t)}{\Phi_{t,\rho(t)}'} A_g\right) \circ \Phi_{t,\rho(t)}^{-1} \right)', \quad g  \in C^1(S^1,\mathbb{R}).
\end{equation}
\end{proposition}
\noindent
The formula in \eqref{Eq:Lineart0} is reminiscent of the one obtained in \cite[Theorem 2.2]{baladi2014linear} for perturbations of expanding circle maps. This is not surprising, since for $t$ approaching 0 the operator is very close to linear. 

At the end of the paper, Appendix \ref{App:Der} gathers some elementary formulae used  throughout our calculations. 


\section{Invariant Densities for the Self-Consistent Transfer Operators} \label{Sec:InvDens}
We now show that the fixed point found above is unique using an invariant cone argument and the Hilbert metric. We remind the reader of our defintion 
 \[
 \mc V_a:=\left\{\phi:S^1\rightarrow \R^+:\;\frac{\phi(x)}{\phi(y)}\le \exp(a|x-y|)\right\}.
 \]
and of the Hilbert metric
 \[
 \theta_a(\phi,\psi)=\log\frac{1}{\alpha_a(\phi,\psi)\alpha_a(\psi,\phi)}
 \]
 where
 \begin{equation}\label{Eq:HilbertMetric}
 \alpha_a(\psi,\phi)=\inf\left\{\frac{\phi(x)}{\psi(x)},\frac{e^{a|x-y|}\phi(x)-\phi(y)}{e^{a|x-y|}\psi(x)-\psi(y)}:\;x\neq y\right\}.
 \end{equation}
It is a known important fact that $P$ is a contraction with respect to this metric (see e.g. \cite{liverani1995decay,Viana}). In the following we show that for suitable values of $t$ and $a>0$, the operator $\mc L_t$ is also a contraction.
The lemmas below gather a few useful inequalities.
\begin{lemma}\label{Lem:IneqSifInv}
There is $\epsilon_0>0$ sufficiently small such that for every $a>0$, every $t\in(-\epsilon_0,\epsilon_0)$, and $\phi,\psi\in C^0$ with $\int\phi=\int\psi=1$,
 \begin{align}
  \left| \Phi_{t,\psi}^{-1}(x)-\Phi_{t,\phi}^{-1}(x) \right|&\leq K|t|\|\psi-\phi\|_{C^0}\label{Ineq:EstInvDif}
  \end{align}
\end{lemma}
\begin{proof}
 Notice that
  \begin{align*}
   \left| \Phi_{t,\psi}(x)-\Phi_{t,\phi}(x) \right|&=|t|\left| \int\partial_1h(x,y)(\psi(y)-\phi(y)) \right|\\
   &\le |t| K\|\psi-\phi\|_{C^0}
     \end{align*}
      By the mean value theorem, for any $x$

 \begin{align*}
\min|\Phi_{t,\psi}'|&\leq \frac{|\Phi_{t,\psi}(\Phi_{t,\psi}^{-1}(x))-\Phi_{t,\psi}(\Phi_{t,\phi}^{-1}(x))|}{|\Phi_{t,\phi}^{-1}(x)-\Phi_{t,\psi}^{-1}(x)|}\\&=\frac{|\Phi_{t,\phi}(\Phi_{t,\phi}^{-1}(x))-\Phi_{t,\psi}(\Phi_{t,\phi}^{-1}(x))|}{|\Phi_{t,\phi}^{-1}(x)-\Phi_{t,\psi}^{-1}(x)|}\\
&\leq\frac{K|t|\|\psi-\phi\|_{C^0}}{|\Phi_{t,\phi}^{-1}(x)-\Phi_{t,\psi}^{-1}(x)|}. 
 \end{align*}
 Since by assumption $\min |\Phi_{t,\phi}'|\ge 1-K|t|>0$ for any $\phi$, \eqref{Ineq:EstInvDif} follows.
\end{proof}

 \begin{lemma}There is $\epsilon_0>0$ sufficiently small such that for every $a>0$ there is $K_a>0$ such that for every $t\in(-\epsilon_0,\epsilon_0)$ and $\phi,\psi\in\mc V_a$ with $\int\phi=\int\psi=1$, the following holds:
 \begin{align}
  \left| \Phi_{t,\psi}^{-1}(x)-\Phi_{t,\phi}^{-1}(x) \right|&\leq K_a|t|\max\left|1-\frac{\phi}{\psi}\right|\label{Eq:Est1'}\\
   \left| \Phi_{t,\phi}^{-1}(x)-\Phi_{t,\phi}^{-1}(y) \right|&\leq (1+K|t|)|x-y| \label{Eq:Est2}\\
    \max\left|1-\frac{\phi}{\psi}\right|&\ge K_a\log\alpha(\phi,\psi)\label{Eq:Est3max}
 \end{align}

 \end{lemma}
 \begin{proof}
The proof of \eqref{Eq:Est1'} follows along the same lines of the proof of Lemma \ref{Lem:IneqSifInv} after noticing that 
\begin{align*}
\left| \int\partial_1h(x,y)(\psi(y)-\phi(y)) \right|&\leq \max{\psi}\cdot K\cdot\max\left|1-\frac{\phi}{\psi}\right|\\
&\leq e^{\frac a2}K\cdot\max\left|1-\frac{\phi}{\psi}\right|.
\end{align*}
 Since  $\psi\in\mc V_a$ and has integral equal to one, $\max\psi\le e^{\frac a2}$. To prove \eqref{Eq:Est2}, notice that for $t$ sufficiently small $|(\Phi^{-1}_{t,\phi})'|<1+Kt$. To prove \eqref{Eq:Est3max}, pick $x_0\in\T$ such that  $\max|1-\frac{\phi}{\psi}|=1-\frac{\phi}{\psi}(x_0)$. Then the inequality is implied by
 \begin{align*}
 \log\alpha(\phi,\psi)\leq \log\left(\frac{\phi}{\psi}(x_0)\right)= \log\left[1-\left(1-\frac{\phi}{\psi}(x_0)\right)\right]&\leq K_a^{-1}\left(1-\frac{\phi}{\psi}(x_0)\right)
 \end{align*}
as $\phi/\psi$ is lower bounded by $e^{-\frac{a}{2}}$.  Analogous conclusion holds if $\max_{x\in\T}|1-\frac{\phi}{\psi}(x)|=\frac{\phi}{\psi}(x_0)-1$ for some $x_0\in\T$.
 \end{proof}
 
 We are now ready to prove Proposition \ref{Lem:ConeCont}.
 \begin{proof}[Proof of Proposition \ref{Lem:ConeCont}]
\textbf{Step 1: There is $\lambda<1$ such that $\mc L_t\mc V_a\subset \mc V_{\lambda a}$}

First of all, we show that there is $K$ such that $Q_t \mc V_a\subset \mc V_{a+K(a+1)|t|}$ for $t$ sufficiently small and $a>0$. In fact, we show more. Pick $\epsilon_0>0$ such that $1- |t|\max| \partial_1 h|>0$ for $t\in[-\epsilon_0,\epsilon_0]$. For any $\phi,\psi\in\mc V_a$
 we have
 \[
 \frac{\phi\circ \Phi_{t,\psi}^{-1}(x)}{\phi\circ \Phi_{t,\psi}^{-1}(y)} \leq e^{a(1+K|t|)}|x-y| \quad \text{and} \quad  \frac{(1+t\partial_xA_\psi)\circ \Phi_{t,\psi}^{-1}(y)}{(1+t\partial_xA_\psi)\circ \Phi_{t,\psi}^{-1}(x)} \leq e^{K|t||x-y|}, 
 \]
 where the bound on the first expression follows from $\phi\in\mc V_a$ and \eqref{Eq:Est2}, while that on the second expression follows from the smoothness of $h$ (in particular $K$ depends on the norms of the derivatives of $h$). Thus
 \[
 \frac{Q_{t,\psi}\phi(x)}{Q_{t,\psi}\phi(y)}=\frac{\phi\circ \Phi_{t,\psi}^{-1}(x)}{\phi\circ \Phi_{t,\psi}^{-1}(y)}\frac{(1+t\partial_xA_\psi)\circ \Phi_{t,\psi}^{-1}(y)}{(1+t\partial_xA_\psi)\circ \Phi_{t,\psi}^{-1}(x)}\\
 \le e^{(a+(a+1)K|t|)|x-y|}.
 \]
 Analogously, it is easy to check (by noticing that $\left| \frac{f'(x)}{f'(y)} \right | \leq e^{\frac{1}{\omega}\max|f''||x-y|}$) that if $\phi\in\mc V_{a}$ then $P\phi\in \mc V_{a/\omega+\max|f''|/\omega^2}$.
 

Picking 
\[
a_0:=\frac{\max|f''|}{\omega(\omega-1)}
\]
one has that for any $a>a_0$,
\[
a/\omega+\max|f''|/\omega^2<a.
\]
 Therefore, for any $a>a_0$ there is $\epsilon_0>0$ and  $a_1,a_2$ satisfying 
 \[
 a<a+(a+1)K|\epsilon_0|=:a_1<a_2<\omega\left(a-\frac{\max |f''|}{\omega^2}\right).
 %
 %
\]

   With this choice of $a_1$ and $a_2$ we have: $Q_t \mc V_a\subset \mc V_{a_1}\subset \mc V_{a_2}$ for $t\in (-\epsilon_0,\epsilon_0)$, and $P\mc V_{a_2}\subset \mc V_{\left(\frac{a_2}{\omega }+\frac{\max|f''|}{\omega^2 }\right) }\subset \mc V_a$. More precisely,  $\mc L_t\mc V_a\subset \mc V_{\lambda a}$ with $\lambda<1$ depending on $a$. \footnote{ If $\mc L_t$ were a linear operator, this would be enough to conclude that $\mc L_t$ is a contraction with respect to the Hilbert metric. Since it is nonlinear, we need to show it directly by estimating the Hilbert distance.}.
 
 From now on, $a>a_0$ is considered to be fixed. 
 \\
{\bf Step 2: for $t$ small, $\mc L_t:\mc V_a\rightarrow \mc V_a$ is a contraction with respect to $\theta_a$}

Let us start by showing how assuming
 \begin{equation}\label{Eq:BndProjQ}\theta_{a_2}(Q_t\phi,Q_t\psi)\leq \left[1+\mc O(|t|)\right]\theta_{a}(\phi,\psi)
 \end{equation} yields the proof of the claim in Step 2.
Recall that with the choice of $a$ and $a_2$ as in Step 1, $P:(\mc V_{a_2},\theta_{a_2})\rightarrow (\mc V_a,\theta_a)$ is a contraction as implied by the  linearity of $P$ and the fact that $\mc V_a\subset \mc V_{a_2}$. Call $\sigma<1$ the contraction rate, i.e. $\theta_a(P\phi,P\psi)\le \sigma\theta_{a_2}(\phi,\psi)$ for any $\phi,\psi\in \mc V_{a_2}$. This and \eqref{Eq:BndProjQ} imply that $\theta_{a}(\mc L_t\phi,\mc L_t\psi)\leq \sigma\left[1+\mc O(t)\right]\theta_{a}(\phi,\psi)$. Picking $|t|$ sufficiently small, equation \eqref{Eq:BndProjQ} implies that $\mc L_t$ is also a contraction.

Now we proceed with the proof of \eqref{Eq:BndProjQ}. Notice that  by  the triangle inequality
\[
\theta_{a_2}(Q_t\phi,Q_t \psi)\leq \theta_{a_2}(Q_{t,\phi}\phi,Q_{t,\phi}\psi)+\theta_{a_2}(Q_{t,\phi}\psi,Q_{t,\psi}\psi).
\]

For the first term,  the operator $Q_{t,\phi}$ is linear and maps the cone $\mc V_a$ to $\mc V_{a_1}$. $\mc V_{a_1}$ is strictly inside the cone $\mc V_{a_2}$. Since $Q_{t,\phi}(\mc V_a)$ is strictly contained in $\mc V_{a_2}$ and its diameter in $(\mc V_{a_2},\theta_{a_2})$ is finite \cite[Proposition 2.5]{Viana}, the mapping $Q_{t,\psi}:(\mc V_a,\theta_a)\rightarrow (\mc V_{a_2},\theta_{a_2})$ is a contraction \cite{Viana}. 
Therefore, 
\[
\theta_{a_2}(Q_{t,\phi}\phi,Q_{t,\phi}\psi)\le  \theta_a(\phi,\psi).
\]

For the second term, computations are lengthier and occupy the rest of the proof. Recall that 
\[\theta_{a_2}(Q_{t,\phi}\psi,Q_{t,\psi}\psi)=\log\frac{1}{\alpha_{a_2}(\phi,\psi)\alpha_{a_2}(\psi,\phi)}\] 
where $\alpha$ is the infimum of two expressions as defined in \eqref{Eq:HilbertMetric}. For the first expression in the definition of $\alpha_{a_2}$:
 \begin{align*}
 \frac{Q_{t,\phi} \psi(x)}{Q_{t,\psi} \psi(x)}
 &=\frac{\psi\circ \Phi_{t,\phi}^{-1}(x)}{\psi\circ \Phi_{t,\psi}^{-1}(x)}\cdot\frac{(1+t\partial_xA_\psi)\circ \Phi_{t,\psi}^{-1}(x)}{(1+t\partial_xA_\phi)\circ \Phi_{t,\phi}^{-1}(x)}=\\
 &=\frac{\psi\circ \Phi_{t,\phi}^{-1}(x)}{\psi\circ \Phi_{t,\psi}^{-1}(x)} \cdot\frac{(1+t\partial_xA_\psi)\circ \Phi_{t,\psi}^{-1}(x)}{(1+t\partial_xA_\psi)\circ \Phi_{t,\phi}^{-1}(x)}\cdot\frac{(1+t\partial_xA_\psi)\circ \Phi_{t,\phi}^{-1}(x)}{(1+t\partial_xA_\phi)\circ \Phi_{t,\phi}^{-1}(x)}
 \end{align*}
The bound on the first two factors follows from $\psi,\phi\in\mc V_a$, $(1+tA_{\psi})\in\mc V_{K|t|}$  (again $K$ depends on bounds on $h$ and its derivatives), and from inequality \eqref{Eq:Est1'} that give
\[
\frac{\psi\circ \Phi_{t,\phi}^{-1}(x)}{\psi\circ \Phi_{t,\psi}^{-1}(x)} \cdot\frac{(1+t\partial_xA_\psi)\circ \Phi_{t,\psi}^{-1}(x)}{(1+t\partial_xA_\psi)\circ \Phi_{t,\phi}^{-1}(x)}\ge  e^{-(a+K|t|)|\Phi_{t,\phi}^{-1}(x)-\Phi_{t,\psi}^{-1}(x)|}\ge e^{-aK_a|t|\max\left|1-\frac{\phi}{\psi}\right|}.
\]
For the third factor
\begin{align}
\frac{(1+t\partial_xA_\psi)(x)}{(1+t\partial_xA_\phi)(x)}&= \frac{1+t\int \partial_1h(x,y)\psi(y)dy}{1+t\int \partial_1h(x,y)\phi(y)dy}\nonumber\\ 
&\ge \frac{1+t\int \partial_1h(x,y)\phi(y)dy+t\left(\min\{\frac{\psi}{\phi}\}-1\right)\int |\partial_1h(x,y)|\phi(y)dy}{{1+t\int \partial_1h(x,y)\phi(y)dy}}\nonumber\\
&\ge 1-K|t|\max\left|\frac{\psi}{\phi}-1\right|\label{Eq:Est3}\\
&\ge \alpha_a(\phi,\psi)^{K|t|}.\nonumber
\end{align}
Putting all the estimates together and using \eqref{Eq:Est3max} one obtains
\[
 \frac{Q_{t,\phi} \psi(x)}{Q_{t,\psi} \psi(x)}\ge \alpha_{a}(\phi,\psi)^{K_a|t|}.
\]
where $K_a>0$ indicates a constant depending on $h$ and on $a$ only.

For the second expression in the definition of $\alpha$ we need to estimate
 \begin{align*}
 \frac{e^{a_2|x-y|}Q_{t,\phi}\psi(x)-Q_{t,\phi}\psi(y)}{e^{a_2|x-y|}Q_{t,\psi}\psi(x)-Q_{t,\psi}\psi(y)}=\frac{e^{a_2|x-y|}\frac{\psi}{(1+tA_\phi)}\circ \Phi_{t,\phi}^{-1}(x)-\frac{\psi}{(1+tA_\phi)}\circ \Phi_{t,\phi}^{-1}(y)}{e^{a_2|x-y|}\frac{\psi}{(1+tA_\psi)}\circ \Phi_{t,\psi}^{-1}(x)-\frac{\psi}{(1+tA_\psi)}\circ \Phi_{t,\psi}^{-1}(y)}= A\cdot B
 \end{align*}
 where
 \begin{align*}
 A&:=\frac{e^{a_2|x-y|}\frac{\psi}{(1+tA_\phi)}\circ \Phi_{t,\phi}^{-1}(x)-\frac{\psi}{(1+tA_\phi)}\circ \Phi_{t,\phi}^{-1}(y)}{e^{a_2|x-y|}\frac{\psi}{(1+tA_\phi)}\circ \Phi_{t,\psi}^{-1}(x)-\frac{\psi}{(1+tA_\phi)}\circ \Phi_{t,\psi}^{-1}(y)}\\
 B&:=\frac{e^{a_2|x-y|}\frac{\psi}{(1+tA_\phi)}\circ \Phi_{t,\psi}^{-1}(x)-\frac{\psi}{(1+tA_\phi)}\circ \Phi_{t,\psi}^{-1}(y)}{e^{a_2|x-y|}\frac{\psi}{(1+tA_\psi)}\circ \Phi_{t,\psi}^{-1}(x)-\frac{\psi}{(1+tA_\psi)}\circ \Phi_{t,\psi}^{-1}(y)}
 \end{align*}
 
 \begin{align}
 A&\ge \frac{e^{a_2|x-y|}\frac{\psi}{(1+tA_\phi)}\circ \Phi_{t,\psi}^{-1}(x)e^{-a_1|\Phi_{t,\psi}^{-1}(x)-\Phi_{t,\phi}^{-1}(x)|}-\frac{\psi}{(1+tA_\phi)}\circ \Phi_{t,\psi}^{-1}(y)e^{a_1|\Phi_{t,\psi}^{-1}(y)-\Phi_{t,\phi}^{-1}(y)|}}{e^{a_2|x-y|}\frac{\psi}{(1+tA_\phi)}\circ \Phi_{t,\psi}^{-1}(x)-\frac{\psi}{(1+tA_\phi)}\circ \Phi_{t,\psi}^{-1}(y)}\nonumber\\
 &\ge \frac{e^{a_2|x-y|}\frac{\psi}{(1+tA_\phi)}\circ \Phi_{t,\psi}^{-1}(x)e^{-K_a'|t| \max\left|1-\frac{\phi}{\psi}\right|}-\frac{\psi}{(1+tA_\phi)}\circ \Phi_{t,\psi}^{-1}(y)e^{K_a'|t| \max\left|1-\frac{\phi}{\psi}\right|}}{e^{a_2|x-y|}\frac{\psi}{(1+tA_\phi)}\circ \Phi_{t,\psi}^{-1}(x)-\frac{\psi}{(1+tA_\phi)}\circ \Phi_{t,\psi}^{-1}(y)}\label{Ineq:A1}\\
 &\ge e^{K_a'|t| \max\left|1-\frac{\phi}{\psi}\right|}\left(1+\frac{e^{-2K'|t| \max\left|1-\frac{\phi}{\psi}\right|}-1}{{e^{a_2|x-y|}\frac{\psi}{(1+tA_\phi)}\circ \Phi_{t,\psi}^{-1}(x)-\frac{\psi}{(1+tA_\phi)}\circ \Phi_{t,\psi}^{-1}(y)}}\right)\label{Ineq:A2}\\
 &\ge e^{K_a''|t| \max\left|1-\frac{\phi}{\psi}\right|}\ge \alpha_{a}(\phi,\psi)^{K_a|t|}\label{Ineq:A3}
 \end{align}
  where \eqref{Ineq:A1} follows from  \eqref{Eq:Est1'},  \eqref{Ineq:A3} follows from the fact that the denominator in \eqref{Ineq:A2} is bounded since $\frac{\phi}{(1+tA_\phi)}\circ \Phi_{t,\psi}^{-1}$ belongs to $\mc V_{a_1}$ with $a_1<a_2$, and $K_a$, $K'_a$ and $K''_a$ are constants depending on $a$. The factor $B$ can be bounded in a similar way: 
 \begin{align*}
 B
 \ge \alpha_{a}(\phi,\psi)^{K_a|t|}.
 \end{align*}
Putting together all the previous estimates one obtains 
\begin{align*}
\alpha_{a_2}(Q_{t,\phi}\psi,Q_{t,\psi} \psi)\ge \alpha_{a}(\phi,\psi)^{2K_a|t|},
\end{align*}  
and analogously
\begin{align*}
\alpha_{a_2}(Q_{t,\psi} \psi,Q_{t,\phi}\psi)\ge \alpha_{a}(\psi,\phi)^{2K_a|t|}
\end{align*}  
which implies
\[
\theta_{a_2}(Q_{t,\phi}\psi,Q_{t,\psi} \psi)=\log\frac{1}{\alpha_{a_2}(Q_{t,\phi}\psi,Q_{t,\psi} \psi)\alpha_{a_2}(Q_{t,\psi} \psi,Q_{t,\phi}\psi)}\leq 2K_a|t|\theta_{a}(\phi,\psi).
\]
 \end{proof}
 We are now ready to conclude the proof of Theorem \ref{Thm:ExistLowReg}
 \begin{proof}[Proof of Theorem \ref{Thm:ExistLowReg}]
 The proof can be finished by standard arguments. The existence of an invariant density $\rho(t)$ in $\Lip(S^1,\mathbb{R}^+)$ follows for example by the argument presented in \cite[Section 2.3]{Viana} and the exponential convergence of $\mc L_t^n\phi$ to $\rho(t)$ in the $C^0$ norm, hence uniqueness of $\rho(t)$ can be argued as in \cite[Section 2.4]{Viana}.
 \end{proof}
From now on let's assume that $f \in C^5(S^1,S^1)$ and $h\in C^5(S^1\times S^1,\mathbb{R})$, and prove point (i) of Theorem \ref{Thm:LinResp}, i.e. we show that in this case $\rho(t)\in C^3(S^1,\R^+)$. Recall that $\|\cdot\|_{C^0}$ denotes the supremum norm on $C^0(S^1,\R)$, and 
\[
\mc C_{\bo C}:=\left\{\phi\in C^4(S^1,\R^+):\;\int \phi=1,\; \left\|\frac{d^i\phi}{dx^i}\right\|_{C^0}<C_i\,\;i=1,...,4\right\}.
\]
\begin{proof}[Proof of Proposition \ref{Prop:InvSetDens}]
Let $t \in [-\varepsilon_0,\varepsilon_0]$ and $\phi\in C^4(S^1,\R^+)$. We remind the reader that $F_{t,\varphi}$ is defined as $F_{t,\varphi}=f \circ \Phi_{t,\varphi}$. By \eqref{Eq:ApNormComp1}--\eqref{Eq:ApNormComp3} of the Appendix and some elementary bounds, we obtain
\begin{align}
	|F_{t,\varphi}'| &\geq  \omega(1-K_0|t|), \label{Eq:FBound1}\\
	|F_{t,\varphi}''| &\leq c((1+K_0|t|)^2+K_1)|t|, \label{Eq:FBound2}\\
	|F_{t,\varphi}'''| &\leq c((1+K_0|t|)^3+(1+K_0|t|)K_1|t|+K_2|t|), \label{Eq:FBound3}
	\end{align}
where $K_i=\max |h^{(i)}|$ and $c=\max\{\max|f'''|,3\max|f''|,\max|f'|\}$.

We now choose the constants $C_i$, $i=1,\dots,4$, starting with the choice of $C_1$. We denote the inverse branches of $F_{t,\varphi}$ by $F_{t,\varphi,i}^{-1}$, $i=1,\dots,N$. 
\begin{align*} 
\left\|(\mc L_t \varphi)' \right\|_{C^0} &=\left\|\left(\sum_{i=1}^N \frac{\varphi}{F_{t,\varphi}'} \circ F_{t,\varphi,i}^{-1}\right)' \right\|_{C^0} = \left\| \sum_{i=1}^N \left(\frac{\varphi'}{(F_{t,\varphi}')^2} - \frac{\varphi F_{t,\varphi}''}{(F_{t,\varphi}')^3}\right)\circ F_{t,\varphi,i}^{-1} \right\|_{C^0} \\
& \leq N \left\| \frac{\varphi'}{(F_{t,\varphi}')^2} - \frac{\varphi F_{t,\varphi}''}{(F_{t,\varphi}')^3} \right\|_{C^0} \\
&\leq N\left(\frac{1}{\omega^2(1-K_0t)^2}\|\varphi'\|_{C^0}+\frac{\max|f''|(1+K_0t)^2+cK_1t}{\omega^3(1-K_0t)^3}\|\varphi\|_{C^0}\right)
\end{align*}
by the bounds \eqref{Eq:FBound1}--\eqref{Eq:FBound2}. Since $\varphi$ is continuous and has unit integral, we have $\|\varphi\|_{C^0} \leq 1 +\| \varphi' \|_{C^0}$ giving us
\begin{align}
\left\|(\mc L_t \varphi)' \right\|_{C^0} &\leq N\left(\frac{1}{\omega^2(1-K_0t)^2}+\frac{\max|f''|(1+Kt)^2+cK_1t}{\omega^3(1-K_0t)^3} \right)\|\varphi'\|_{C^0} \nonumber \\
&\hspace{0.4cm} +\frac{N(\max|f''|(1+K_0t)^2+cK_1t)}{\omega^3(1-K_0t)^3} \nonumber \\
&=: \sigma_1 \|\varphi'\|_{C^0} + R_1. \label{Eq:LY1}
\end{align}
By the assumption \eqref{Eq:Assumf} on $f$, we can choose $t$ small enough (by potentially decreasing $\varepsilon_0$) such that $\sigma_1 < 1$. Choose $C_1$ such that
\[
\frac{R_1}{1-\sigma_1} \leq C_1,
\]  
and then $\|\varphi'\|_{C^0} \leq C_1$ implies $\left\|(\mc L_t \varphi)' \right\|_{C^0} \leq C_1$. 

We move on to the choice of $C_2$.

\begin{align} 
\left\|(\mc L_t \varphi)'' \right\|_{C^0} &=\left\|\left(\sum_{i=1}^N \frac{\varphi}{F_{t,\varphi}'} \circ F_{t,\varphi,i}^{-1}\right)'' \right\|_{C^0} \nonumber \\
& \leq N \left\| \frac{\varphi''}{(F_{t,\varphi}')^3} - 3\frac{\varphi' F_{t,\varphi}''}{(F_{t,\varphi}')^4}-\frac{\varphi F_{t,\varphi}'''}{(F_{t,\varphi}')^4}+3\frac{\varphi (F_{t,\varphi}'')^2}{(F_{t,\varphi}')^5} \right\|_{C^0} \nonumber \\
& \leq N\frac{1}{\omega^3(1-K_0t)^3}\|\varphi''\|_{C^0}+N\left\|3\frac{ F_{t,\varphi}''}{(F_{t,\varphi}')^4}+\frac{ F_{t,\varphi}'''}{(F_{t,\varphi}')^4}-3\frac{ (F_{t,\varphi}'')^2}{(F_{t,\varphi}')^5} \right\|_{C^0}\|\varphi'\|_{C^0} \nonumber \\
&\hspace{0.4cm}+N\left\| \frac{ F_{t,\varphi}'''}{(F_{t,\varphi}')^4}-3\frac{ (F_{t,\varphi}'')^2}{(F_{t,\varphi}')^5} \right\|_{C^0} \nonumber \\
&=: \sigma_2\|\varphi''\|_{C^0} +R_2^{(1)}\|\varphi'\|_{C^0}+R_2^{(2)}. \label{Eq:LY2}
\end{align}
For sufficiently small $t$ we have $\sigma_2 < 1$ by assumption \eqref{Eq:Assumf} and by \eqref{Eq:FBound1}--\eqref{Eq:FBound3} $R_2^{(1)}$, $R_2^{(2)}$ are bounded uniformly in $\varphi$ and in $t$ if $t$ is in some fixed neighborhood of zero, so choosing  
\[
\frac{R_2^{(1)}C_1+R_2^{(2)}}{1-\sigma_2} \leq C_2
\] 
will be sufficient.

Similar calculations can be done for the third and fourth derivative. We can conclude that
\begin{align*} 
\left\|(\mc L_t \varphi)''' \right\|_{C^0} &\leq \sigma_3\|\varphi'''\|_{C^0} +R_3^{(1)}\|\varphi''\|_{C^0}+R_3^{(2)}\|\varphi'\|_{C^0}+R_3^{(3)} \\
\left\|(\mc L_t \varphi)'''' \right\|_{C^0} &\leq \sigma_4\|\varphi''''\|_{C^0} +R_4^{(1)}\|\varphi'''\|_{C^0}+R_4^{(2)}\|\varphi''\|_{C^0}+R_4^{(3)}\|\varphi'\|_{C^0}+R_4^{(4)}
\end{align*}
where $\sigma_3,\sigma_4 < 1$ for sufficiently small $t$ and $R_i^{(j)}$ only depend on the derivatives of $f$ and $h$ and on $t$ (such that they are bounded uniformly in $t$ if $t$ is in some fixed neighborhood of zero). Choosing 
\begin{align*}
\frac{R_3^{(1)}C_2+R_3^{(2)}C_1+R_3^{(3)}}{1-\sigma_3} &\leq C_3 \\
\frac{R_4^{(1)}C_3+R_4^{(2)}C_2+R_4^{(3)}C_1+R_4^{(4)}}{1-\sigma_4} &\leq C_4
\end{align*}
will be sufficient for our purposes.
\end{proof}

\begin{lemma}\label{Prop:ContSelfconOperator}
If $f \in C^5(S^1,S^1)$ and $h \in C^5(S^1\times S^1,\R)$, then for $1\leq k \leq 3$, $Q_t: \mc C_{\bo C} \rightarrow \mc C_{\bo C}$ is continuous with respect to the $C^k$ distance.
\end{lemma}
\begin{proof}
 Fix $k\in\{1,2,3\}$. For any $\phi,\psi\in C^k(S^1,\R)$, using expression \eqref{Eq:QepsDef} for $Q_t$ one obtains
\begin{align*}
Q_t\phi-Q_t\psi&=\frac{\phi-\psi}{1+tA'_{\phi}}\circ \Phi^{-1}_{t,\phi}(x)+\\
&\quad+\left[\frac{\psi}{1+tA'_{\phi}}-\frac{\psi}{1+tA'_{\psi}}\right]\circ \Phi^{-1}_{t,\phi}(x)+ \\
&\quad+\frac{\psi}{1+tA'_{\psi}}\circ \Phi^{-1}_{t,\phi}(x)-\frac{\psi}{1+tA'_{\psi}}\circ \Phi^{-1}_{t,\psi}(x).
\end{align*}
For $t$ sufficiently, small $\frac{1}{1+tA'_\phi}$ and $\Phi^{-1}_{t,\phi}$ belong to $C^4$.
Then it follows that \begin{align*}
\left\|\frac{\phi-\psi}{1+tA'_{\phi}}\circ \Phi^{-1}_{t,\phi}(x)\right\|_{C^k}&\leq M \left\|\frac{1}{1+tA'_{\phi}}\circ \Phi^{-1}_{t,\phi}(x) \right\|_{C^k}\left\|(\phi-\psi)\circ \Phi^{-1}_{t,\phi}(x)\right\|_{C^k} \\
&\leq M' \left\|\frac{1}{1+tA'_{\phi}} \right\|_{C^k}\left(\max_{\ell=1,\dots,k}\|\Phi^{-1}_{t,\phi}\|_{C^k}^{\ell}\right)^2\left\|\phi-\psi\right\|_{C^k} \\
&\leq \tilde K\|\phi-\psi\|_{C^k},
\end{align*}
where the first two inequalities hold by \eqref{Eq:ApNormProd} and \eqref{Eq:ApNormComp} of the Appendix, respectively. The third inequality can be proved by using \eqref{Eq:ApNormInv} of the Appendix and some standard bounds on the derivatives of $\Phi_{t,\varphi}$. 
To bound the second term, notice that
\[
\frac{1}{1+tA'_\phi}-\frac{1}{1+tA'_\psi}=t\frac{A'_\psi-A'_\phi}{(1+tA'_\phi)(1+tA'_\psi)}
\]
and $\|A'_\psi-A'_\phi\|_{C^k}\le K\|\psi-\phi\|_{C^0}$. Then similarly to the calculation above,
\begin{align*}
&\left\|\left[\frac{\psi}{1+tA'_{\phi}}-\frac{\psi}{1+tA'_{\psi}}\right]\circ \Phi^{-1}_{t,\phi}(x)\right\|_{C^k}\\
&\leq M \left\|t\frac{\psi}{(1+tA'_\phi)(1+tA'_\psi)}\circ \Phi_{t,\varphi}^{-1}(x)\right\|_{C^k} \left\|(\phi-\psi)\circ \Phi^{-1}_{t,\phi}(x)\right\|_{C^k} \\
&\leq M' \left\|t\frac{\psi}{(1+tA'_\phi)(1+tA'_\psi)}\right\|_{C^k} \left(\max_{\ell=1,\dots,k}\|\Phi^{-1}_{t,\phi}\|_{C^k}^{\ell}\right)^2 \left\|\phi-\psi\right\|_{C^k} \\
&\leq \tilde K_{\mathbf{C}} \|\phi-\psi\|_{C^k},\end{align*}
where the last inequality also exploits that $\psi \in \mc C_{\mathbf{C}}$. Using \eqref{Eq:ApNormProd1}--\eqref{Eq:ApNormProd4} of the Appendix we can compute that there exists an explicit polynomial function $p$ such that
 \begin{align*}
&\left\|\frac{\psi}{1+tA'_{\psi}}\circ \Phi^{-1}_{t,\phi}(x)-\frac{\psi}{1+tA'_{\psi}}\circ \Phi^{-1}_{t,\psi}(x)\right\|_{C^k} \\
&\leq  p\left(\left\|\frac{\psi}{1+tA'_{\psi}}\right\|_{C^{k+1}},\|\Phi^{-1}_{t,\phi}\|_{C^k},\|\Phi^{-1}_{t,\psi}\|_{C^k}\right)\|\phi-\psi\|_{C^k}.
\end{align*}
Since $\frac{\psi}{1+tA'_{\psi}} \in C^4$, simple calculations as before give us
\[
\left\|\frac{\psi}{1+tA'_{\psi}}\circ \Phi^{-1}_{t,\phi}(x)-\frac{\psi}{1+tA'_{\psi}}\circ \Phi^{-1}_{t,\psi}(x)\right\|_{C^k} \leq \tilde{K}_{\mathbf{C}}\|\phi-\psi\|_{C^k}.
\]
By triangle inequality one can bound $\|Q_t\phi-Q_t\psi\|_{C^k}$ and the thesis follows.
\end{proof}
We can finally prove that $\rho(t) \in C^3(S^1,\mathbb{R})$.
\begin{proof}[Proof of Theorem \ref{Thm:LinResp} Point (i)]
$\mc C_{\bo C}\subset C^4(S^1,\R)$ is  convex and relatively compact in $C^3(S^1,\R)$. The invariance of the set $C_{\bo C}$ under the action of $\mc L_t$, as claimed in Proposition \ref{Prop:InvSetDens} and the continuity of $\mc L_t$ restricted to $C_{\bo C}$ as stated in Lemma \ref{Prop:ContSelfconOperator} implies by the Schauder fixed point theorem that $\mc L_t$ has a fixed density in $C^3(S^1,\R)$. By Theorem \ref{Thm:ExistLowReg} this density must be the unique Lipschitz density $\rho(t)$. 
\end{proof}

\section{Smooth Dependence of the Invariant Densities on the Coupling Strength} \label{Sec:SmoothDep}
Given $t \in [-\varepsilon_1,\varepsilon_1]$, the relation $t\mapsto\rho(t)$ defines a curve of fixed densities, that is, for each $t$ the density $\rho(t)$ is such that $\mc L_t\rho(t)=\rho(t)$. Recall the definition of $\hat{\mc C}_{\epsilon,K_1,K_2}$
\[
\hat {\mc C}_{\epsilon,K_1,K_2}:=\left\{ \gamma:[-\epsilon, \epsilon]\rightarrow \mc C_{\bo C}\mbox{ s.t. }  \;\substack{(t,x)\mapsto \gamma(t)(x) \mbox{ is } C^2\\ \sup_t\left\|\frac{d}{dt}\gamma(t)\right\|_{C^3}<K_1, \;\sup_t\left\|\frac{d^2}{dt^2}\gamma(t)\right\|_{C^2}<K_2}\right\}
\]
and the operator $\mc L$  \[(\mc L\gamma)(t)=\mc L_t(\gamma (t))\mbox{ for all }t\in [-\epsilon,\epsilon]\]  
 for  $\gamma\in \hat {\mc C}_{\epsilon,K_1,K_2}$.
The strategy consists of (i) showing that  there are $\epsilon_2$, $K_1$, and $K_2$ such that $\mc L$ keeps $\hat {\mc C}_{\epsilon_2,K_1,K_2}$ invariant, and  (ii) using Proposition \ref{Lem:ConeCont} to show that this implies that $t\mapsto \rho(t)$ is $C^1$.  
\begin{proof}[Proof of Proposition \ref{Prop:ChatInv}]
Let us start by noticing that the implicit function theorem implies that the function $g:[-\epsilon_1, \epsilon_1]\times S^1\rightarrow S^1$, defined as $g(t,y):= (\Phi_{t,\phi})^{-1}(y)$, is differentiable and $C^5$.  Furthermore, its $C^5$ norm can be uniformly bounded with respect to $\phi$ \footnote{This is a consequence of the fact that the function to which we apply the implicit function theorem, $F(t,y,x)=\Phi_{t,\phi}(x)-y$, has $C^5$ norm uniformly bounded varying $\phi$.}. Let $\epsilon \leq \varepsilon_1$ and pick $\gamma\in\hat{\mc C}_{\epsilon,K_1,K_2}$.
\begin{align}\label{Eq:FirstDev}
\frac{d}{dt}(\mc L\gamma)(t)&=\left(\frac{d}{ds}\mc L_s\right)_{s=t}\gamma(t)+\lim_{\delta\rightarrow 0}\frac{\mc L_t\gamma(t+\delta)-\mc L_t\gamma(t)}{\delta}.
\end{align}
For the first term in the right hand side of \eqref{Eq:FirstDev}
\begin{align}
\left(\frac{d}{ds}\mc L_s\right)_{s=t}\phi&=P\left(\lim_{\delta\rightarrow 0}\frac{ Q_{t+\delta}-Q_{t}}{\delta}\phi\right) \nonumber \\
&= P\frac{d}{ds}\left(\frac{\phi}{\partial_x\Phi_{s,\phi}}\circ \Phi^{-1}_{s,\phi}\right)_{s=t} \nonumber\\
&=P\bigg(\frac{\phi'}{\partial_x\Phi_{t,\phi}}\circ \Phi^{-1}_{t,\phi}\partial_t\Phi^{-1}_{t,\phi} \nonumber \hspace{0.4cm}-\frac{\phi}{(\partial_x\Phi_{t,\phi})^2} \circ \Phi^{-1}_{t,\phi}\partial_t[\partial_x\Phi_{t,\phi}\circ \Phi^{-1}_{t,\phi}]\bigg) \nonumber\\
&=P\bigg(\frac{\phi'}{\partial_x\Phi_{t,\phi}}\circ \Phi^{-1}_{t,\phi}\partial_t\Phi^{-1}_{t,\phi}\nonumber  \\
&\hspace{0.4cm}-\frac{\phi}{(\partial_x\Phi_{t,\phi})^2} \circ \Phi^{-1}_{t,\phi}[\partial_t\partial_x\Phi_{t,\phi}\circ \Phi^{-1}_{t,\phi}+\partial_x^2\Phi_{t,\phi}\circ \Phi^{-1}_{t,\phi}\partial_t\Phi^{-1}_{t,\phi}]\bigg). \label{Eq:ExpExpdds}
\end{align}
Therefore, putting $\phi=\gamma(t)$ and recalling that $\gamma(t)\in  \mc C_{\bo C}$, for $t$ sufficiently small all the functions above have bounded norm in $C^3(S^1,\R)$, and therefore there is a constant $\tilde K$ depending on $f,h$ and $\bo C$ such that
\begin{equation}\label{Eq:BndDer1}
\left\|\left(\frac{d}{ds}\mc L_s\right)_{s=t}\gamma(t)\right\|_{C^3}\leq \tilde K.
\end{equation}
The second term in the RHS of \eqref{Eq:FirstDev} equals
\begin{align*}
\Xi(t)&:=P\lim_{\delta\rightarrow 0}\frac{Q_t\gamma(t+\delta)-Q_t\gamma(t)}{\delta}\\&=  P\lim_{\delta\rightarrow 0}\frac{1}{\delta}\left[\frac{\gamma(t+\delta)}{1+tA_{\gamma(t+\delta)}'}\circ \Phi^{-1}_{t,\gamma(t+\delta)}-\frac{\gamma(t)}{1+tA_{\gamma(t)}'}\circ \Phi^{-1}_{t,\gamma(t)}\right]\\
 &=P\lim_{\delta\rightarrow 0}\frac{\gamma(t+\delta)-\gamma(t)}{\delta(1+A'_{\gamma(t+\delta)})}\circ \Phi^{-1}_{t,\gamma(t+\delta)}\\
 &\hspace{0.4cm}+P\lim_{\delta\rightarrow 0}\frac{1}{\delta}\left[\gamma(t)\left(\frac{1}{1+tA'_{\gamma(t+\delta)}}-\frac{1}{1+tA'_{\gamma(t)}}\right)\right]\circ\Phi^{-1}_{t,\gamma(t+\delta)}\\
 &\hspace{0.4cm}+P\lim_{\delta\rightarrow 0}\frac{1}{\delta}\left[\frac{\gamma(t)}{1+tA'_{\gamma(t)}}\circ \Phi^{-1}_{t,\gamma(t+\delta)}-\frac{\gamma(t)}{1+tA'_{\gamma(t)}}\circ \Phi^{-1}_{t,\gamma(t)}\right],
 \end{align*}
 giving
 \begin{align}
 &\Xi(t)= \nonumber\\
 &= PQ_{t,\gamma(t)}\frac{d}{dt}\gamma(t)\label{Eq:Term0}\\
 &+P\left[\left(\gamma(t)\frac{d}{ds}\left(\frac{1}{1+tA'_{\gamma(s)}}\right)_{s=t}\right)\circ \Phi^{-1}_{t,\gamma(t)}+\partial_x\left( Q_{t,\gamma(t)}\gamma(t)\right)\frac{d}{ds} \left(\Phi^{-1}_{t,\gamma(s)}\right)_{s=t}\right]\label{Eq:Term1}
\end{align}
For the first term of the sum, one can argue in a way similar to the proof of Proposition \ref{Prop:InvSetDens} and obtain the inequality
\begin{equation} \label{Eq:LYdtgamma}
\left\|PQ_{t,\gamma(t)}\frac{d}{dt}\gamma(t)\right\|_{C^3}\leq \sigma\left\|\frac{d}{dt}\gamma(t)\right\|_{C^3}+\tilde K
\end{equation}
with $\sigma<1$ and $\tilde K>0$ (in general different from the $\tilde K$ above). To see this, notice that in the referenced calculations we use uniform bounds for $F^{(k)}_{t,\varphi}$ in $\varphi$, only exploiting its unit integral. Furthermore, since $\int \frac{d}{dt}\gamma(t)=0$, the bound $\|\varphi\|_{C^0} \leq 1+\|\varphi'\|_{C^0}$ simplifies to $\|\frac{d}{dt}\gamma(t)\|_{C^0} \leq \|(\frac{d}{dt}\gamma(t))'\|_{C^0}$.

The regularity assumptions allow to deduce that the $C^3$ norm of the other terms is uniformly bounded. With computations similar to those in the proof of Proposition \ref{Prop:InvSetDens}, one can deduce  there are $\sigma<1$ (by possibly decreasing $\varepsilon_1$) and $\tilde K>0$ such that $\|\Xi\|_{C^3}\leq \sigma\|\frac{d}{dt}\gamma\|_{C^3}+\tilde K$. Putting this estimate together with the one in \eqref{Eq:BndDer1} we get
\begin{align*}
\left\|\frac{d}{dt}(\mc L\gamma)(t)\right\|_{C^3}&\leq \left\|\left(\frac{d}{ds}\mc L_s\right)_{s=t}\gamma(t)\right\|_{C^3}+\|\Xi(t)\|_{C^3}\\
&\leq \sigma\left\|\frac{d}{dt}\gamma(t)\right\|_{C^3}+\tilde K.
\end{align*}
We conclude that there is a $K_1$ sufficiently large such that if $\|\frac{d}{dt}\gamma(t)\|_{C^3}< K_1$ for all $t$ values sufficiently small, $\|\frac{d}{dt}(\mc L\gamma)(t)\|_{C^3}\leq K_1$.

Now, assuming that $\|\frac{d}{dt}\gamma\|_{C^3}< K_1$, we carry out similar computations in order to bound $\|\frac{d^2}{dt^2}(\mc L\gamma)(t)\|_{C^2}$.
\begin{align*}
\frac{d^2}{dt^2}(\mc L\gamma)(t)&=\frac{d}{dt}\left(\left(\frac{d}{ds}\mc L_s\right)_{s=t}\gamma(t)\right)+\frac{d}{dt}\Xi(t)
\end{align*}

The first term corresponds to the derivative with respect to $t$ of \eqref{Eq:ExpExpdds} when $\phi=\gamma(t)$, which in turn  is a combination of the functions 
\begin{align*}
&\partial_x\left(\frac{\gamma(t)}{\partial_x\Phi_{t,\gamma(t)}}\right)\circ \Phi^{-1}_{t,\gamma(t)},\; \partial_s\Phi^{-1}_{s,\gamma(t)}|_{s=t}, \\
&\frac{\gamma(t)}{(\partial_x\Phi_{t,\gamma(t)})^2}\circ \Phi^{-1}_{t,\gamma(t)},\; \partial_s\partial_x\Phi_{s,\gamma(t)}|_{s=t}\circ \Phi^{-1}_{t,\gamma(t)},\; \partial_x^2\Phi_{t,\gamma(t)}\circ \Phi^{-1}_{t,\gamma(t)}.
\end{align*} 
With the given regularity hypotheses all the derivatives with respect to $t$ of the above terms are continuous in the variable $x$. In particular, for the first term notice that the assumed regularity of all the terms implies
\[
\frac{d}{dt}\partial_x\left(\frac{\gamma(t)}{\partial_x\Phi_{t,\gamma(t)}}\right)=\partial_x\frac{d}{dt}\left(\frac{\gamma(t)}{\partial_x\Phi_{t,\gamma(t)}}\right).
\] 
We conclude that there exists a $\tilde K > 0$ such that $\left\|\frac{d}{dt}\left(\left(\frac{d}{ds}\mc L_s\right)_{s=t}\gamma(t)\right)\right\|_{C^2} < \tilde K$.

As for $\frac{d}{dt}\Xi(t)$, consider the decomposition of $\Xi(t)$ given by \eqref{Eq:Term0}--\eqref{Eq:Term1}. The derivative of $PQ_{t,\gamma(t)}\frac{d}{dt}\gamma(t)$ is
\begin{align*}
\frac{d}{ds}\left(PQ_{s,\gamma(s)}\frac{d}{ds}\gamma(s)\right)_{s=t}&= P\lim_{\delta\rightarrow 0}\frac{(Q_{t+\delta,\gamma(t+\delta)}-Q_{t,\gamma(t+\delta)})\frac{d}{dt}\gamma(t+\delta)}{\delta}\\ 
&\hspace{0.4cm}+P\lim_{\delta\rightarrow 0}\frac{(Q_{t,\gamma(t+\delta)}-Q_{t,\gamma(t)})\frac{d}{dt}\gamma(t+\delta)}{\delta}\\&\hspace{0.4cm}+P\lim_{\delta\rightarrow 0}\frac{Q_{t,\gamma(t)}(\frac{d}{dt}\gamma(t+\delta)-\frac{d}{dt}\gamma(t))}{\delta}.
\end{align*}

The first term in the sum on the RHS can be treated with computations similar to those leading to \eqref{Eq:ExpExpdds}, and equals
\begin{align*}
&P\lim_{\delta\rightarrow 0}\frac{(Q_{t+\delta,\gamma(t)}-Q_{t,\gamma(t)})\frac{d}{dt}\gamma(t)}{\delta}\\&=P\left(\partial_x\left(\frac{\frac{d\gamma(t)}{dt}}{\partial_x\Phi_{t,\gamma(t)}}\right)\circ \Phi^{-1}_{t,\gamma(t)}\partial_s\Phi^{-1}_{s,\gamma(t)}|_{s=t}\right.\\
&\left. -\frac{\frac{d\gamma(t)}{dt}}{(\partial_x\Phi_{t,\gamma(t)})^2} \circ \Phi^{-1}_{t,\gamma(t)}\left[\partial_s\partial_x\Phi_{s,\gamma(t)}|_{s=t}\circ \Phi^{-1}_{t,\gamma(t)}+\partial_x^2\Phi_{t,\gamma(t)}\circ \Phi^{-1}_{t,\gamma(t)}\partial_s\Phi^{-1}_{s,\gamma(t)}|_{s=t}\right]\right)
\end{align*}
which can be verified to have uniformly bounded $C^2$ norm. 

The  second term can be expressed as
\begin{align}
P\lim_{\delta\rightarrow 0}\frac{(Q_{t,\gamma(t+\delta)}-Q_{t,\gamma(t)})\frac{d}{dt}\gamma(t+\delta)}{\delta}&=\frac{d}{ds}\left(\frac{\frac{d\gamma}{dt}(t)}{\partial_x\Phi_{t,\gamma(s)}}\circ \Phi_{t,\gamma(s)}\label{Eq:secondtermsecderiv}^{-1}\right)_{s=t}
\end{align}
and noticing that 
\begin{align*}
\frac{d}{ds}\frac{1}{\partial_x\Phi_{t,\gamma(s)}}=-\frac{1}{(\partial_x\Phi_{t,\gamma(s)})^{2}}\frac{d}{ds}\partial_x\Phi_{t,\gamma(s)}, \quad
\frac{d}{ds}\left(\frac{d\gamma(t)}{dt}\right),\quad\mbox {and}\;
\partial_{s}\Phi_{t,\gamma(s)}^{-1}
\end{align*}
have uniformly bounded $C^2$ norm, it follows from straightforward computations that the whole of \eqref{Eq:secondtermsecderiv} has bounded norm. \footnote{Notice that in the expression above the factor $\partial_x\left(\frac{d\gamma}{d t}\right)$ appears, and it is for this kinds of factors that the choice of controlling  two different norms for $\frac{d\gamma}{dt}$ and $\frac{d^2\gamma}{dt^2}$ becomes crucial.}

For the third term we can find an inequality with a similar argument that leads to \eqref{Eq:LYdtgamma}:
\[
\left\|PQ_{t,\gamma(t)}\frac{d^2}{dt^2\gamma(t)}\right\|_{C^2}\leq\sigma\left\|\frac{d^2}{dt^2\gamma(t)}\right\|_{C^2}+\tilde K
\] 
For what concerns the derivative of \eqref{Eq:Term1}, it involves the same terms listed above. Most of the factors appearing have been previously treated apart from
\[
\frac{d}{ds}\left(\frac{1}{1+tA'_{\gamma(s)}}\right)_{s=t}\quad\quad\mbox{and} \quad\quad\partial_x\left( Q_{t,\gamma(t)}\gamma(t)\right).
\]
For the first one is easy to see that all the derivatives in $t$ of $1+tA'_{\gamma(s)}$ have  bounded $C^2$ norm in $x$. For the second one, it follows from previous computations that $Q_{t,\gamma(t)}\gamma(t)$ is at least $C^1((-\epsilon,\epsilon),S^1)$ and this implies
\[
\frac{d}{dt}\partial_xQ_{t,\gamma(t)}\gamma(t)=\partial_x\frac{d}{dt}Q_{t,\gamma(t)}\gamma(t).
\]
$\frac{d}{dt}Q_{t,\gamma(t)}\gamma(t)$ can be shown to have uniformly bounded $C^3$ norm in the variable $x$ with the same computations used to control \eqref{Eq:FirstDev}, and therefore $\|\partial_x\frac{d}{dt}Q_{t,\gamma(t)}\gamma(t)\|_{C^2}\le \tilde K$.   
Therefore, 
\[
\left\|\frac{d^2}{dt^2}(\mc L\gamma)(t)\right\|_{C^2}\leq \sigma \left\|\frac{d^2}{dt^2}\gamma(t)\right\|_{C^2}+\tilde K 
\]
so we can find $K_2$ sufficiently large so that $\|\frac{d}{dt}\gamma(t)\|_{C^3}<K_1$, with $K_1$ as above, and $\|\frac{d^2}{dt^2}\gamma(t)\|_{C^2}<K_2$ implies $\|\frac{d^2}{dt^2}(\mc L\gamma)(t)\|_{C^2}<K_2$. 
\end{proof}

\begin{proof}[Proof of Theorem \ref{Thm:LinResp} Point (ii)]
Given any $\gamma\in \hat{\mc C}_{\epsilon_2,K_1,K_2}$, it follows from Proposition \ref{Prop:ChatInv} that  $\{\mc L^n\gamma\}_{n\in\N_0}\in \hat{\mc C}_{\epsilon_2,K_1,K_2}$. Since $ \hat{\mc C}_{\epsilon_2,K_1,K_2}\subset C^2([-\epsilon_2,\epsilon_2], C^2(S^1,\R))$ is relatively compact in $ C^1([-\epsilon_2,\epsilon_2], C^1(S^1,\R))$
\footnote{Recall that $C^2(S^1,\R)$ is compactly embedded in $C^1(S^1,\R)$, and thus given a bounded ball $B'\subset C^2(S^1,\R)$ and any $\delta>0$ there are $\mc F_\delta:=\{\phi_i\}_{i=1}^{N_\delta}\subset C^1(S^1,\R)$ such that the union of the $\delta/2$ $C^1$-balls centered at the functions in $\mc F_\delta$ covers $B'$. Fix a bounded ball $B\subset C^2([-\epsilon_2,\epsilon_2], C^2(S^1,\R))$, there is  $B'\subset C^2(S^1,\R)$ such that $\gamma([-\epsilon_2,\epsilon_2])\subset B'$ for any $\gamma\in B$. Now fix $\delta>0$ and pick points $\{t_i\}_{i=1}^M$ partitioning $ [-\epsilon_2,\epsilon_2]$ such that $\|\gamma(s_1)-\gamma(s_2)\|_{C^0}<\delta/2$ whenever $s_1,s_2\in[t_i,t_{i+1}]$ for some $i$. Then, for any $(i_1,...,i_M)\in\{1,...,N\}^M$ fix a $C^1$ function $\gamma_{(i_1,...,i_M)}:[-\epsilon_2,\epsilon_2]\rightarrow C^1(S^1,\R)$ such that $\gamma_{(i_1,...,i_M)}(t_j)=\phi_{i_j}$. It is then easy to check that for every $\gamma\in B$ there is $(i_1,...,i_M)\in\{1,...,N\}^M$ such that $\gamma$ is at $C^1$ distance no more than $\delta$ from $\gamma_{(i_1,...,i_M)}$.},
 there is a strictly increasing sequence $n_k$ and $\bar\gamma\in C^1([-\epsilon_2,\epsilon_2],C^1(S^1,\R))$ such that $\mc L^{n_k}\gamma\rightarrow \bar\gamma$ in the $C^1$ topology.  In particular, for every $t$, $\|\mc L_t^{n_k}\gamma(t)-\bar\gamma(t)\|_{C^0}\rightarrow 0$ for $k\rightarrow\infty$. However, from Proposition \ref{Lem:ConeCont} we know that for every $t$, $\|\mc L_t^{n}\gamma(t)-\rho(t)\|_{C^0}\rightarrow 0$ for $n\rightarrow\infty$ from which we deduce that $\bar\gamma(t)=\rho(t)$, and therefore $t\mapsto \rho(t)$ belongs to $C^1([-\epsilon_2,\epsilon_2], C^1(S^1,\R))$.
\end{proof}

\section{Linear Response Formula} \label{Sec:LinRespFormula}
In this section we derive a linear response formula for the curve $t \mapsto \rho(t)$ of fixed densities. 
We proceed to prove the proposition in an analogous way to \cite[Theorem 2.2]{baladi2014linear} and apply perturbation arguments to deal with the self-consistency. 


\begin{lemma} \label{lem:transferderhat}
	Considering $\{P_t\}_{t \in [-\varepsilon_2,\varepsilon_2]}$ as bounded operators from $C^2(S^1,\mathbb{R})$ to $C^1(S^1,\mathbb{R})$, we claim that the map $t \mapsto P_t$ is differentiable at $\hat{t} \in (-\varepsilon_2,\varepsilon_2)$ and 
	\[
	\partial_{t}P_t|_{t=\hat{t}}(\varphi)=-P\left(\left(\frac{\varphi}{\Phi_{\hat{t},\rho(\hat{t})}'} (A_{\rho(\hat{t})}+\hat{t}A_{\partial_t{\rho(t)}|_{t=\hat{t}}})\right) \circ \Phi_{\hat{t},\rho(\hat{t})}^{-1} \right)'.
	\] 	
\end{lemma} 
\begin{proof}
	Recall that $P_t=PQ_{t,\rho(t)}$ which implies that
	\begin{equation} \label{eq:decomp}
		\partial_{t}P_t=P\partial_{t}Q_{t,\rho(t)}.
	\end{equation}
	Recall also that for $t\in[-\epsilon_0,\epsilon_0]$, $\Phi_{t,\rho(t)}$ is a  diffeomorphism of $S^1$ and $t \mapsto \Phi_{t,\rho(t)}$ is $C^2$. 
	Let $\varphi \in C^2$.
	\begin{align*}
		&Q_{t,\rho(t)}(\varphi)-Q_{\hat{t},\rho(\hat{t})}(\varphi)=\\
		&=\frac{\varphi}{\Phi_{t,\rho(t)}'} \circ \Phi_{t,\rho(t)}^{-1}-\frac{\varphi}{\Phi_{\hat{t},\rho(\hat{t})}'} \circ \Phi_{\hat{t},\rho(\hat{t})}^{-1}\\
		&=-(t-\hat{t}) (\partial_{t}\Phi_{t,\rho(t)}|_{t=t^*} \circ \Phi_{t^*,\rho(t^*)}^{-1})'\left(\frac{\varphi}{\Phi_{t^*,\rho(t^*)}'} \circ \Phi_{t^*,\rho(t^*)}^{-1}\right)\\
		&-(t-\hat{t}) \partial_{t}\Phi_{t,\rho(t)}|_{t=t^*} \circ \Phi_{t^*,\rho(t^*)}^{-1}\left(\frac{\varphi'}{(\Phi_{t^*,\rho(t^*)}')^2} \circ \Phi_{t^*,\rho(t^*)}^{-1} -\frac{\varphi\cdot \Phi_{t^*,\rho(t^*)}''}{(\Phi_{t^*,\rho(t^*)}')^3} \circ \Phi_{t^*,\rho(t^*)}^{-1}\right).
	\end{align*}
	for some $t^* \in (t, \hat{t})$. Notice that
	\begin{align*}
	&(\partial_t\Phi_{t,\rho(t)}|_{t=t^*}\circ\Phi_{t^*,\rho(t^*)}^{-1})'\left(\frac{\varphi}{\Phi_{t^*,\rho(t^*)}'} \circ \Phi_{t^*,\rho(t^*)}^{-1}\right) \\
	&=\left(\frac{\varphi}{(\Phi_{t^*,\rho(t^*)}')^2} \circ \Phi_{t^*,\rho(t^*)}^{-1}\right)\partial_t\Phi_{t,\rho(t)}'|_{t=t^*}\circ\Phi_{t^*,\rho(t^*)}^{-1}
	\end{align*}
	implying that
	\begin{align*}
		\partial_{t}Q_{t,\rho(t)}|_{t=\hat{t}}(\varphi)=&-\left(\frac{\varphi}{(\Phi_{\hat{t},\rho(\hat{t})}')^2} \partial_t\Phi_{t,\rho(t)}'|_{t=\hat{t}}\right)\circ\Phi_{\hat{t},\rho(\hat{t})}^{-1} \\
		&-\left(\left(\frac{\varphi'}{(\Phi_{\hat{t},\rho(\hat{t})}')^2} -\frac{\varphi\cdot \Phi_{\hat{t},\rho(\hat{t})}''}{(\Phi_{\hat{t},\rho(\hat{t})}')^3}\right) \cdot \partial_{t}\Phi_{t,\rho(t)}|_{t=\hat{t}}\right) \circ \Phi^{-1}_{\hat{t},\rho(\hat{t})} \\
		=&-\left(\left(\frac{\varphi}{\Phi_{\hat{t},\rho(\hat{t})}'} (\partial_t\Phi_{ t,\rho(t)})_{t= \hat{t}}\right) \circ\Phi_{\hat{t},\rho(\hat{t})}^{-1}\right)'
	\end{align*}
	and by \eqref{eq:decomp}
	\begin{align*}
		\partial_{t}P_t|_{t=\hat{t}}(\varphi)&=-P\left(\left(\frac{\varphi}{\Phi_{\hat{t},\rho(\hat{t})}'} (\partial_t\Phi_{ t,\rho(t)})_{t= \hat{t}}\right) \circ\Phi_{\hat{t},\rho(\hat{t})}^{-1}\right)'\\
		&=-P\left(\left(\frac{\varphi}{\Phi_{\hat{t},\rho(\hat{t})}'} (A_{\rho(\hat{t})}+\hat{t}A_{\partial_t{\rho(t)}|_{t=\hat{t}}})\right) \circ \Phi_{\hat{t},\rho(\hat{t})}^{-1} \right)'.
	\end{align*}
\end{proof}
\noindent Notice that this lemma implies that
\[
\partial_{t}P_{t}|_{t=\hat{t}}\rho(\hat{t})=-P[\Kappa_{\hat{t}}(\rho(\hat{t}))+\hat{t}\Kappa_{\hat{t}}(\partial_{t}\rho|_{t=\hat{t}})],
\]
where $\Kappa_{\hat{t}}$ is given by \eqref{Eq:ThetaOp}.

\begin{proof}[Proof of Proposition \ref{Prop:LinearResp}]
	Recall that $P_t$ is the transfer operator of $f\circ \Phi_{t,\rho(t)}$ which is an at least  $C^3$ transformation of $S^1$. Furthermore, there is $K>0$ such that $|\Phi_{t,\rho(t)}'|>1-Kt$ and thus   $f\circ \Phi_{t,\rho(t)}$ is uniformly expanding for any $t$ with $|t|<\epsilon$ for some sufficiently small $\epsilon$. Standard results on uniformly expanding maps (see \cite{boyarsky2012laws}) imply that $P_t$ has a spectral gap in $C^2(S^1,\R)$. In particular, 1 is an isolated and simple eigenvalue and $\R\rho(t)$ is the corresponding eigenspace. Because of the uniform Lasota-Yorke type inequalities \eqref{Eq:LY1}--\eqref{Eq:LY2}, a uniform bound can be given on the spectral gap of all $P_t$ such that $t$ is small enough. This implies that we can find a positively oriented curve $\gamma$ on the complex plane around 1 such that no other element of the spectrum of any $P_t$ is contained inside of it, and the projection formula
	\begin{equation} \label{eq:int}
		\rho(t)=\frac{1}{2\pi i}\int_{\gamma}(z-P_t)^{-1}\varphi(z)\text{d}z
	\end{equation}
	holds for any $\varphi \in C^2$ such that $\int_{S^1}\varphi=1$. One can prove this using the decomposition of $\mathcal{L}_{t,\rho(t)}$ as stated in \cite[Theorem 7.1.1]{boyarsky2012laws} (exploiting the fact that $F_{t,\rho(t)}$ is $C^2$ and uniformly expanding which properties by \cite{krzyzewski1969invariant} imply that $F_{t,\rho(t)}$ mixing) and a residue computation.
	
	For $z \in \gamma$ we have 
	\begin{equation} \label{eq:conj}
		(z-P_{\hat{t}})^{-1}-(z-P_{t})^{-1}=(z-P_{\hat{t}})^{-1}(P_{\hat{t}}-P_{t})(z-P_{t})^{-1}
	\end{equation}
	where we view $(z-P_{t})^{-1}$ as an operator acting on $C^2(S^1,\R)$, $(P_{\hat{t}}-P_{t})$ as an operator from $C^2(S^1,\R)$ to $C^1(S^1,\R)$ and $(z-P_{\hat{t}})^{-1}$ as acting on $C^1(S^1,\R)$.   
	
	Letting $t \to {\hat{t}}$ in \eqref{eq:conj} by Lemma \ref{lem:transferderhat} we have that
	\begin{equation} \label{eq:resolvder}
		\partial_{t}(z-P_{t})^{-1}|_{t=\hat{t}}=(z-P_{\hat{t}})^{-1}\partial_{t}P_{t}|_{t=\hat{t}}(z-P_{\hat{t}})^{-1}.
	\end{equation}
	Substituting $\varphi=\rho(\hat{t})$ in \eqref{eq:int} and differentiating both sides we get
	\begin{align}
		\partial_{t}\rho(t)|_{t=\hat{t}}&=\frac{1}{2\pi i}\int_{\gamma}(z-P_{\hat{t}})^{-1}\partial_{t}P_{t}|_{t=0}(z-P_{\hat{t}})^{-1}\rho(\hat{t})(z)\text{d}z \nonumber \\
		&=\frac{1}{2\pi i}\int_{\gamma}(z-P_{\hat{t}})^{-1}\partial_{t}P_{t}|_{t=\hat{t}}\rho(\hat{t})(z)\frac{1}{z-1}\text{d}z \nonumber \\
		&=-\frac{1}{2\pi i}\int_{\gamma}(z-P_{\hat{t}})^{-1}P(\Kappa_{\hat{t}}(\rho(\hat{t}))+\hat{t}\Kappa_{\hat{t}}(\partial_{t}\rho(t)|_{t=\hat{t}}))(z)\frac{1}{z-1}\text{d}z \nonumber\\
		&=-(1-P_{\hat{t}})^{-1}P(\Kappa_{\hat{t}}(\rho(\hat{t}))+\hat{t}\Kappa_{\hat{t}}(\partial_{t}\rho(t)|_{t=\hat{t}})), \label{eq:res4}
	\end{align} 
	where the last step is a residue computation.
	
	Thus we can write
	\[
	(1-P_{\hat{t}}+\hat{t}P\Kappa_{\hat{t}})\partial_{t}\rho(t)|_{t=\hat{t}}=-P\Kappa_{\hat{t}}(\rho(\hat{t})).
	\]
	It is a well-known fact that invertible operators form an open set in the space of bounded linear operators between normed spaces. More precisely, if an operator $T$ is invertible, then any operator $S$ such that $\|T-S\| \leq \|T^{-1}\|^{-1}$ is also invertible. Since $1-P_{\hat{t}}$ is invertible as an operator from $C^2(S^1,\mathbb{R})$ to $C^1(S^1,\mathbb{R})$, the invertibility of $1-P_{\hat{t}}+\hat{t}P\Kappa_{\hat{t}}$ follows if we are able to choose $\hat{t}$ so small that
	\[
	\|\hat{t}P\Kappa_{\hat{t}}\|_{C^2 \to C^1} \leq \|(1-P_{\hat{t}})^{-1}\|_{C^2 \to C^1} ^{-1} \quad \Leftrightarrow \quad \hat{t} \leq \frac{1}{\|P\Kappa_{\hat{t}}\|_{C^2 \to C^1} \|(1-P_{\hat{t}})^{-1}\|_{C^2 \to C^1} }.
	\]
	For this it suffices to show that $t \mapsto \|P\Kappa_{t}\|_{C^2 \to C^1} \|(1-P_{t})^{-1}\|_{C^2 \to C^1}$ is uniformly bounded in a small interval around zero. The norm $\|(1-P_{t})^{-1}\|_{C^2 \to C^1}$ is finite for all sufficiently small $t$ (because of the above mentioned spectral gap), so it can be uniformly bounded in an interval $[-\varepsilon_2,\varepsilon_2]$. As for $\Kappa_t$, we can compute that 
	\begin{align*}
		&\sup_{x \in S^1}|\Kappa_t{\varphi}(x)|=\\
		&=\sup_{x \in S^1}\left|\left(\left(\frac{\rho(t)}{\Phi_{t,\rho(t)}'} A_\varphi\right) \circ \Phi_{t,\rho(t)}^{-1}(x) \right)'\right| \\
		&= \sup_{x \in S^1} \left|\frac{\rho'(t)A_{\varphi}\Phi_{t,\rho(t)}'+\rho(t)A'_{\varphi}\Phi_{t,\rho(t)}'-\rho(t)A_{\varphi}\Phi_{t,\rho(t)}''}{(\Phi_{t,\rho(t)}')^2} \circ \Phi_{t,\rho(t)}^{-1}(x) \cdot (\Phi_{t,\rho(t)}^{-1})'(x)\right| \\
		& \leq \frac{1}{(1-tK)^3}\left(K(1+tK)\sup_{x \in S^1}|\rho'(t)|+(K(1+tK)+tK^2)\sup_{x \in S^1}|\rho(t)|\right)
	\end{align*}
	and a similar formula for $\sup_{x \in S^1}|(\Kappa_t{\varphi})
	'(x)|$ depending smoothly on $\rho,\rho',\rho''$ and $t$. Since $(t,x) \mapsto \rho(t)(x)$ is $C^2$, $\|\Kappa_{t}\|_{C^2 \to C^1}$ is uniformly bounded for $t \in [-\varepsilon_2,\varepsilon_2]$. 
	
	Thus we can choose $\varepsilon_3$ further decreasing $\varepsilon_2$ such that $1-P_{\hat{t}}+\hat{t}P\Kappa_{\hat{t}}$ is invertible for $t \in [-\varepsilon_3,\varepsilon_3]$ and
	\[
	\partial_{t}\rho(t)|_{t=\hat{t}}=-(1-P_{\hat{t}}+\hat{t}P\Kappa_{\hat{t}})^{-1}P\Kappa_{\hat{t}}(\rho(\hat{t})).
	\]
\end{proof}

\appendix
 \section{Appendix} \label{App:Der}
In this appendix we recall some simple formulae on derivatives that have been used extensively throughout the document. First of all higher derivatives of a product: Assume $\Phi,\Psi\in C^4$, then
 \begin{align}
 (\Phi\Psi)'&=\Phi'\Psi+\Phi\Psi'\label{Eq:ApNormProd1}\\
  (\Phi\Psi)''&=\Phi''\Psi+2\Phi'\Psi'+\Phi\Psi''\\
  (\Phi\Psi)'''&=\Phi'''\Psi+3\Phi''\Psi'+3\Psi'\Phi''+\Psi\Phi'''\\
  (\Phi\Psi)''''&=\Phi''''\Psi+4\Phi'''\Psi'+6\Phi''\Psi''+4\Phi'\Psi'''+\Psi\Phi'''' \label{Eq:ApNormProd4}
 \end{align}
 which implies that there is a constant $M > 0$ such that
 \begin{equation}\label{Eq:ApNormProd}
 \|\Phi\Psi\|_{C^k}\le M\|\Phi\|_{C^k}\|\Psi\|_{C^k}, \qquad k=1,\dots,4.
 \end{equation}
 Higher derivatives of a composition: Assume $\Phi,\Psi\in C^4$, then
 \begin{align}
 (\Phi\circ \Psi)'&=\Phi'\circ\Psi\Psi'\label{Eq:ApNormComp1}\\
 (\Phi\circ \Psi)''&=\Phi''\circ\Psi(\Psi')^2+\Phi'\circ\Psi\Psi''\\
  (\Phi\circ \Psi)'''&=\Phi'''\circ\Psi(\Psi')^3+2\Phi''\circ\Psi\Psi'\Psi''+\Phi''\circ\Psi\Psi''\Psi'+\Phi'\circ\Psi\Psi''' \label{Eq:ApNormComp3}\\
  (\Phi\circ \Psi)''''&=\Phi''''\circ\Psi(\Psi')^4+\Phi'''\circ\Psi \left[3(\Psi')^2\Psi''+\Psi''(\Psi')^2+2(\Psi')^2\Psi''\right]+\nonumber\\
  &\quad+\Phi''\circ\Psi\left[3(\Psi'')^2+4\Psi'\Psi'''\right]+\Phi'\circ\Psi(\Psi'''')
 \end{align}
 which implies that there is a numerical constant $M$ such that 
 \begin{equation}\label{Eq:ApNormComp}
\|\Phi\circ \Psi\|_{C^k}\leq M \|\Phi\|_{C^k}\cdot\max_{\ell=1,...,k}\|\Psi\|_{C^k}^\ell, \qquad k=1,\dots, 4.
 \end{equation}
 
We also give the formulae for the higher derivatives of the inverse function. Let $\Phi\in C^4$ be invertible and such that $\min|\Phi'|>0$, then
\begin{align}
(\Phi^{-1})'&=\frac{1}{\Phi'}\circ \Phi^{-1}\label{Eq:ApNormInv1}\\
(\Phi^{-1})''&=-\frac{\Phi''}{(\Phi')^2}\circ\Phi^{-1}(\Phi^{-1})'=-\frac{\Phi''}{(\Phi')^3}\circ\Phi^{-1}\\
(\Phi^{-1})'''&=\left[2\frac{(\Phi'')^2}{(\Phi')^3}-\frac{\Phi'''}{(\Phi')^2}\right]\circ\Phi^{-1}(\Phi^{-1})'-\frac{\Phi''}{(\Phi')^2}\circ\Phi^{-1}(\Phi^{-1})''\nonumber\\
&=\left[2\frac{(\Phi'')^2}{(\Phi')^4}-\frac{\Phi'''}{(\Phi')^3}-\frac{(\Phi'')^2}{(\Phi')^5}\right]\circ\Phi^{-1}\\
(\Phi^{-1})''''&=\left[4\frac{\Phi''\Phi'''}{(\Phi')^5}-8\frac{(\Phi'')^3}{(\Phi')^6}-\frac{\Phi''''}{(\Phi')^4}+3\frac{\Phi'''\Phi''}{(\Phi')^5}-2\frac{\Phi'''}{(\Phi')^6}+5\frac{(\Phi'')^3}{(\Phi')^7}\right]\circ\Phi^{-1}
\end{align}
 and so there exists an $M > 0$ such that
 \begin{equation}\label{Eq:ApNormInv}
  \|\Phi^{-1}\|_{C^k}\leq M\cdot\max_{n=1,...,7}(\min|\Phi'|)^{-n}\cdot\max_{\ell=1,...,3}\|\Phi\|_{C^k}^\ell, \qquad k=1,\dots,4.
 \end{equation}
 
 \newpage

\bibliographystyle{acm}
\bibliography{references_linresp}
\nocite{*}
 
\end{document}